\documentclass[reqno]{amsart}
\usepackage{amsmath}
\usepackage{amssymb,stmaryrd,mathrsfs}
\usepackage{esint}

\newcommand{\ud}[0]{\mathrm{d}}
\newcommand{\dist}[0]{\operatorname{dist}}

\newcommand{\abs}[1]{|#1|}
\newcommand{\Babs}[1]{\Big|#1\Big|}

\newcommand{\Norm}[2]{\|#1\|_{#2}}

\newcommand{\BNorm}[2]{\Big\|#1\Big\|_{#2}}
\newcommand{\pair}[2]{\langle #1,#2 \rangle}
\newcommand{\Bpair}[2]{\Big\langle #1,#2 \Big\rangle}
\newcommand{\ave}[1]{\langle #1\rangle}

\newcommand{\spt}[0]{\operatorname{spt}}
\newcommand{\loc}[0]{\operatorname{loc}}
\newcommand{\id}[0]{\operatorname{id}}
\newcommand{\sign}[0]{\operatorname{sgn}}

\newcommand{\R}{\mathbb{R}}

\newcommand{\N}{\mathbb{N}}
\newcommand{\Z}{\mathbb{Z}}

\newcommand{\Exp}[0]{\mathbb{E}}
\newcommand{\good}[0]{\operatorname{good}}

\numberwithin{equation}{section}

\newtheorem{theorem}{Theorem}[section] 
\newtheorem{lemma}[theorem]{Lemma}     
\newtheorem{corollary}[theorem]{Corollary}
\newtheorem{proposition}[theorem]{Proposition}

\theoremstyle{definition}
\newtheorem{conjecture}{Conjecture}
\newtheorem{definition}{Definition}
\newtheorem{remark}{Remark}
\newtheorem{assumption}{Assumption}

\makeatletter
\@namedef{subjclassname@2010}{%
  \textup{2010} Mathematics Subject Classification}
\makeatother

\title[Hilbert transform with general measures]{The two-weight inequality for\\ the Hilbert transform with general measures}
\author[T.~P.\ Hyt\"onen]{Tuomas P.\ Hyt\"onen}
\address{Department of Mathematics and Statistics, P.O.B.~68 (Gustaf H\"all\-str\"omin katu~2b), FI-00014 University of Helsinki, Finland}
\email{tuomas.hytonen@helsinki.fi}

\thanks{The author was supported by the European Union through the ERC Starting Grant ``Analytic--probabilistic methods for borderline singular integrals'', and by the Academy of Finland, grants 130166 and 133264 and the Centre of Excellence in Analysis and Dynamics Research.}

\begin{document}

\maketitle

\begin{abstract}
The two-weight inequality for the Hilbert transform is characterized for an arbitrary pair of positive Radon measures $\sigma$ and $w$ on $\R$. In particular, the possibility of common point masses is allowed, lifting a restriction from the recent solution of the two-weight problem by Lacey, Sawyer, Shen and Uriarte-Tuero. Our characterization is in terms of Sawyer-type testing conditions and a variant of the two-weight $A_2$ condition, where $\sigma$ and $w$ are integrated over complementary intervals only. A key novely of the proof is a two-weight inequality for the Poisson integral with `holes'.
\end{abstract}

\tableofcontents

\section{Introduction}

The two-weight problem (in the nowadays preferred dual-weight formulation) for the Hilbert transform
\begin{equation*}
  H(f\ud\sigma)(x)=\int_{\R}\frac{1}{x-y}f(y)\ud\sigma(y),\qquad x\notin\spt f,\quad\spt f\text{ compact},
\end{equation*}
 is to characterize the pairs of non-negative weight functions $\sigma,w\in L^1_{\loc}(\R)$ for which the weighted measures $\ud\sigma:=\sigma\ud x$ and $\ud w:=w\ud x$ satisfy
\begin{equation*}
   \Norm{H(f\ud\sigma)}{L^2(w)}\leq C\Norm{f}{L^2(\sigma)}\qquad\forall\ f\in L^2(\sigma).
\end{equation*}
A variant of this problem was solved in a complex-variable framework by Cotlar and Sadosky \cite{CoSa:79} as early as 1979, but obtaining a real-variable characterization turned out to be a much more difficult problem. A lot of deep work in the last ten years or more was devoted to obtaining conditional characterizations under increasingly general side conditions on $\sigma$ and $w$, including the doubling property \cite[Theorem  15.1]{Volberg:book} and so-called `pivotal' \cite{NTV:2wHilbert} and `energy' \cite{LSU:2wHilbertEnergy} hypotheses. In a culmination of these efforts, a complete solution was finally obtained by Lacey, Sawyer, Shen and Uriarte-Tuero \cite{Lacey:2wHilbert,LSSU:2wHilbert}. Indeed, their recent result goes beyond the weighted problem as stated, in that their characterization is equally valid for any pair of positive Radon measures $\sigma$ and $w$, subject only to the restriction that they do not share a common point mass; this manifestly covers the case of weighted measures.

The aim of this paper is to lift this last remaining side condition, and to establish the following unconditional characterization for any pair of Radon measures $\sigma$ and $w$ on the real line: (A rigorous reading of the theorem requires an interpretation of the meaning of the Hilbert transform in the considered generality, and we return to this technical point in Section~\ref{ss:interpretation} below.)

\begin{theorem}\label{thm:main}
Let $\sigma$ and $w$ be two Radon measures on $\R$. Then the Hilbert transform satisfies the norm inequality
\begin{equation}\label{eq:Hbounded}
  \Norm{H(f\ud\sigma)}{L^2(w)}\leq C\Norm{f}{L^2(\sigma)}\qquad\forall\ f\in L^2(\sigma),
\end{equation}
if and only if it satisfies the global testing conditions
\begin{equation}\label{eq:HtestGlobal}
  \Norm{H(1_I\ud\sigma)}{L^2(w)} \leq \mathcal{H}_{\operatorname{glob}}\sigma(I)^{1/2}, \qquad
  \Norm{H(1_I\ud w)}{L^2(\sigma)} \leq \mathcal{H}_{\operatorname{glob}}^* w(I)^{1/2},
\end{equation}
if and only if it satisfies the local (or Sawyer-type) testing conditions
\begin{equation}\label{eq:HtestLocal}
  \Norm{1_I H(1_I\ud\sigma)}{L^2(w)} \leq {\mathcal{H}}\sigma(I)^{1/2}, \qquad
  \Norm{1_I H(1_I\ud w)}{L^2(\sigma)} \leq {\mathcal{H}}^* w(I)^{1/2},
\end{equation}
and the following version of the Muckenhoupt--Poisson two-weight $A_2$ conditions:
\begin{equation}\label{eq:newA2}
  \sigma(I)\int_{I^c}\frac{\ud w(x)}{(x-c_I)^2}\leq([\sigma,w]_{A_2}^*)^{2},\qquad
  w(I)\int_{I^c}\frac{\ud\sigma(x)}{(x-c_I)^2}\leq ([w,\sigma]_{A_2}^*)^2,
\end{equation}
where inequalities \eqref{eq:HtestGlobal}, \eqref{eq:HtestLocal} and \eqref{eq:newA2} are required for all finite intervals $I\subset\R$.

Moreover, the best constants in these inequalities satisfy
\begin{equation*}
  C\eqsim\mathcal{H}_{\operatorname{glob}}+\mathcal{H}_{\operatorname{glob}}^*
    \eqsim\mathcal{H}+\mathcal{H}^*+[\sigma,w]_{A_2}^*+[w,\sigma]_{A_2}^*.
\end{equation*}
\end{theorem}

Theorem~\ref{thm:main} unifies the recent two-weight theory of Lacey et al.~\cite{Lacey:2wHilbert,LSSU:2wHilbert} with the theory of discrete Hilbert transforms, where even the most classical example  $\sigma=w=\sum_{k\in\Z}\delta_k$ (for which the Hilbert transform is certainly bounded) is not, as such, covered by \cite{Lacey:2wHilbert,LSSU:2wHilbert}. For more general point distributions on the line, Theorem~\ref{thm:main} is probably new, and of course it allows for many further examples that combine continuous and discrete measures in nontrivial ways. 

Aside from the proof (to be discussed below), a key novelty of the above characterization itself is the nature of the $A_2$ conditions \eqref{eq:newA2}, which are weaker than the ones used in \cite{Lacey:2wHilbert,LSSU:2wHilbert}. In the above conditions, the measures $\sigma$ and $w$ are integrated over complementary sets, which is in contrast with the na\"ive two-weight generalization of the classical $A_2$ condition,
\begin{equation}\label{eq:oldA2}
  \sup_I\frac{\sigma(I)w(I)}{\abs{I}^2}<\infty.
\end{equation}
In the one-weight case (which in our dual-weight formulation corresponds to $\sigma(x)=1/w(x)$), this alone is a necessary and sufficient condition for the inequality \eqref{eq:Hbounded}, and much of the two-weight theory has been guided by attempts to find suitably strengthened versions of \eqref{eq:oldA2}. Here, in a sense, we go against this trend by weakening the $A_2$ conditions that were used in the previous developments. This is necessary since, as pointed out in \cite[Sec.~1.2.1]{Lacey:2wHilbertPrimer}, \eqref{eq:oldA2} cannot hold if there is a point $a$ with $\sigma\{a\}w\{a\}>0$, since the measure ratio blows up as $I$ shrinks to $a$. Applying a similar limiting procedure to our condition~\eqref{eq:newA2}, we find that
\begin{equation*}
  \sigma\{a\}\int_{\R\setminus\{a\}}\frac{\ud w(x)}{(x-a)^2}\leq ([\sigma,w]_{A_2}^*)^2.
\end{equation*}
Thus our Muckenhoupt--Poisson condition forces a relatively strong decay for $w$ in the \emph{punctured} neighbourhood of any atom of $\sigma$; by Theorem~\ref{thm:main}, this is necessary for the boundedness of the Hilbert transform. However, the possibility that $\sigma\{a\}w\{a\}>0$ is by no means excluded, as it should not.

Already in the absence of common point masses, we feel that our approach provides useful insight into the existing two-weight theory, in that our conditions~\eqref{eq:HtestLocal} and \eqref{eq:newA2} have more clearly defined roles in controlling the Hilbert transform on and off the support of the original function, respectively. Incidentally, an $A_2$ condition of a similar nature has appeared in the work of Belov, Mengestie and Seip \cite[Eq. (1.5)]{BMS} on discrete Hilbert transforms on sparse sequences in the plane, and we strongly feel that working with this type of weight conditions, rather than \eqref{eq:oldA2} or its stronger versions, will be instrumental in the further development of the two-weight theory. Indeed, following the circulation of the present paper, this framework and some related techniques have already been adopted to study a number of higher dimensional extensions of the two-weight inequalities, in  \cite{LSSUW:Cauchy,LaWick:Riesz,SSU:T1,SSU:fractional}.

The issue of common point masses was already discussed at some length by Lacey \cite[Sec.~1.2.1]{Lacey:2wHilbertPrimer}, who stated (although only in Version~1 of the cited e-print) the following conjecture, formulated in terms of  the Poisson integral
\begin{equation*}
  P(f\ud\sigma,I):=\int_{\R}\frac{\abs{I}f(x)\ud\sigma(x)}{\abs{I}^2+(x-c_I)^2}.
\end{equation*}

\begin{conjecture}[(Lacey~\cite{Lacey:2wHilbertPrimer})]
Let $\sigma$ and $w$ be two Radon measures on $\R$. For each interval $I$, let $b_I$ be the unique (if it exists) point such that
\begin{equation*}
  P(1_{\{b_I\}}\ud\sigma,I) >\frac12 P(\ud\sigma,I),\qquad
  P(1_{\{b_I\}}\ud w,I)>\frac12 P(\ud w,I).
\end{equation*}
Let $\tilde\sigma_I:=\sigma-\sigma\{b_I\}\delta_{b_I}$ if $b_I$ exists, and $\tilde\sigma_I:=\sigma$ otherwise, and let $\tilde w_I$ be defined similarly.

Then the Hilbert transform satisfies the norm inequality \eqref{eq:Hbounded}, if and only if it satisfies the local testing conditions \eqref{eq:HtestLocal} and the following $A_2$-type condition:
\begin{equation}\label{eq:LaceyA2}
  ([\sigma,w]_{A_2}^{**})^2:=\sup_I\Big(  P(\ud\tilde\sigma_I,I)P(\ud w,I)+P(\ud\tilde w_I,I)P(\ud\sigma,I) \Big) <\infty.
\end{equation}
\end{conjecture}

The deeper ``if'' part of Lacey's conjecture is a corollary to Theorem~\ref{thm:main} by means of the following. (For the ``only if'' part, see the proof of \cite[Proposition 3.1]{Lacey:2wHilbertPrimer}.)

\begin{lemma}
The condition \eqref{eq:LaceyA2} implies \eqref{eq:newA2}. More precisely,
\begin{equation*}
  [\sigma,w]_{A_2}^*+[w,\sigma]_{A_2}^*\lesssim  [\sigma,w]_{A_2}^{**}.
\end{equation*}
\end{lemma}

\begin{proof}
If $b_I\in I$, then
\begin{equation*}
   P(\ud\tilde w_I,I)P(\ud\sigma,I) \geq P(1_{I^c}\ud w,I)P(1_I\ud\sigma,I),
\end{equation*}
while if $b_I\notin I$ (including the case that there is no $b_I$), then
\begin{equation*}
   P(\ud\tilde\sigma_I,I)P(\ud w,I) \geq P(1_{I}\ud\sigma,I)P(1_{I^c} w,I).
\end{equation*}
Finally, we observe that
\begin{equation*}
  P(1_I\ud\sigma,I)P(1_{I^c}\ud w,I)\eqsim\frac{\sigma(I)}{\abs{I}}\int_{I^c}\frac{\abs{I}\ud w(x)}{(x-c_I)^{2}}
    =\sigma(I)\int_{I^c}\frac{\ud w(x)}{(x-c_I)^{2}}.
\end{equation*}
This shows that $[\sigma,w]_{A_2}^*\lesssim[\sigma,w]_{A_2}^{**}$, and the other bound follows by symmetry.
\end{proof}

While the formulation of Lacey's $A_2$ condition \eqref{eq:LaceyA2} suggests an approach that specifically addresses the distracting point masses, our present approach to Theorem~\ref{thm:main} hardly mentions the point masses at all, being carefully set up in such a way that their possible presence never makes a difference in our considerations. We feel that this gives certain `robustness' to the argument.

Given that most of the existing two-weight theory was developed in the presence of (at least) the classical $A_2$ condition \eqref{eq:oldA2}, we need to rewrite a major part of the argument of Lacey et al.~\cite{Lacey:2wHilbert,LSSU:2wHilbert}, setting up things in such a way that only complementary intervals appear, whenever making estimates with the $A_2$ condition. For some parts of the proof, like those derived from the `standard' non-homogeneous $T(1)$ theory, this modification is relatively straightforward; for some others, it requires more effort.

One of the key novelties of our proof is a two-weight inequality for the Poisson integral with `holes', as in the conditions~\eqref{eq:newA2}, where the integrals extend over the complement of $I$ only. A need for something like this in the general case was already foreseen by Lacey \cite[Sec.~1.2.1]{Lacey:2wHilbertPrimer}, while in the absence of common point masses, Sawyer's classical two-weight inequality for the full Poisson integral \cite{Sawyer:2wFractional} could be applied as a black box ingredient of the proof. Here, we handle the Poisson integral with holes by a new method which should also be of interest as an alternative approach to Sawyer's classical inequality. On the other hand, it turns out that we only need to elaborate on the first part of the work of Lacey et al. \cite{LSSU:2wHilbert}, the `global-to-local' reduction; the innovative verification of the local estimate by Lacey \cite{Lacey:2wHilbert} is already set up in such a way that it is ready for us to borrow as a black box.

\subsubsection*{Acknowledgements}
I would like to thank Henri Martikainen, Emil Vuorinen, and the anonymous referee for constructive comments that improved the manuscript.

\section{Set-up and preparations}\label{sec:setup}


\subsection{Interpretation of the two-weight Hilbert transform}\label{ss:interpretation}

For a number of reasons, we prefer a `minimalistic' definition of the Hilbert transform $H(\,\cdot\,\ud\sigma)$; cf.~\cite[paragraph of Eq.~(15.6)]{Volberg:book}: informally, it is any linear operator with the property that
\begin{equation*}
   H(f\ud\sigma)(x)=\int_{\R}\frac{1}{x-y}f(y)\ud\sigma(y)\qquad x\notin\spt f,
\end{equation*}
for appropriate functions $f$. The above property does not determine $H$ uniquely, and we will of course need to specify the domain of this linear operator. We turn to these details now.

Let $\mathscr{D}^0$ be a given system of half-open (say, of the form $[a,b)$) dyadic intervals on $\R$, for simplicity chosen so that 
neither $\sigma$ nor $w$ has a point mass in any of the countably many end-points of the intervals of $\mathscr{D}^0$. We call these the dyadic points. 

Let $\mathscr{F}$ be the space of all finite linear combinations of indicators $1_I$, where $I\in\mathscr{D}^0$. Notice that the indicator of any half-open interval with dyadic end-points is in $\mathscr{F}$. We assume that, a priori, we are given a bilinear form $B:\mathscr{F}\times\mathscr{F}\to\R$ with the property that
\begin{equation}\label{eq:Bdisjoint}
  B(f,g)=\iint\frac{1}{x-y}f(y)g(x)\ud\sigma(y)\ud w(x) 
\end{equation}
for all $f,g\in\mathscr{F}$ supported on \emph{disjoint} intervals $I\in\mathscr{D}^0$. It might seem a bit unusual that we postulate the kernel representation even for functions whose supports are not necessarily \emph{separated} (by a positive distance); we return to this point in Remark~\ref{rem:disjVsSep} below.

Observe that this formula does not specify $B(f,g)$ for all $f,g\in\mathscr{F}$, but we require that $B$ is nevertheless defined on all pairs of $f$ and $g$ in such a way that the above formula holds for disjointly supported functions. In different applications, such a bilinear form may naturally arise in different ways, so we prefer this axiomatic approach. For concreteness, we will indicate in Section \ref{subsec:truncated} below how to construct such a bilinear form with the help of the truncated Hilbert transforms, but we stress that this is just one possible way of defining the Hilbert transform, and another approach is sketched in Remark \ref{rem:anotherConstr} at the end of that section.

Now, let us interpret Theorem~\ref{thm:main} in our axiomatic framework.

By the statement that the Hilbert transform satisfies the norm inequality \eqref{eq:Hbounded}, we understand that
\begin{equation}\label{eq:Bbounded}
  \abs{B(f,g)}\leq C\Norm{f}{L^2(\sigma)}\Norm{g}{L^2(w)}
\end{equation}
for all $f,g\in\mathscr{F}$. Since $\mathscr{F}$ is dense in both $L^2(\sigma)$ and $L^2(w)$, this implies the existence of a unique extension of $B$ to a bounded bilinear form $B:L^2(\sigma)\times L^2(w)\to\R$, which satisfies \eqref{eq:Bbounded} with the same constant. By standard Hilbert space theory, such a bilinear form is uniquely associated with a bounded linear operator, say $H(\,\cdot\,\ud\sigma):L^2(\sigma)\to L^2(w)$, and its adjoint, say $H^*(\,\cdot\,\ud w):L^2(w)\to L^2(\sigma)$ such that
\begin{equation*}
  B(f,g)=\pair{H(f\ud\sigma)}{g}_w=\pair{f}{H^*(g\ud w)}_\sigma,
\end{equation*}
where $\pair{\ }{\ }_\sigma$ and $\pair{\ }{\ }_w$ are the inner products in $L^2(\sigma)$ and $L^2(w)$, respectively. Since the adjoint $H^*(\,\cdot\,\ud w)$ is also a Hilbert transform in the above axiomatic sense, except for a minus sign, we henceforth write $-H(\,\cdot\,\ud w)$ for $H^*(\,\cdot\,\ud w)$.

The testing conditions \eqref{eq:HtestGlobal} and \eqref{eq:HtestLocal} admit similar interpretations. The global testing condition \eqref{eq:HtestGlobal} is interpreted as
\begin{equation}\label{eq:BtestGlobal}
  \abs{B(1_I,g)}\leq\mathcal{H}_{\operatorname{glob}}\sigma(I)^{1/2}\Norm{g}{L^2(w)},\qquad
  \abs{B(f,1_I)}\leq\mathcal{H}_{\operatorname{glob}}^*\Norm{f}{L^2(\sigma)}w(I)^{1/2},
\end{equation}
and the local \eqref{eq:HtestLocal} as
\begin{equation}\label{eq:BtestLocal}
  \abs{B(1_I, 1_I g)}\leq\mathcal{H}\sigma(I)^{1/2}\Norm{g}{L^2(w)},\qquad
  \abs{B(1_I f,1_I)}\leq\mathcal{H}^*\Norm{f}{L^2(\sigma)}w(I)^{1/2},
\end{equation}
which are both required for all $f,g\in\mathscr{F}$ and all finite half-open intervals $I$ with dyadic end-points. The $A_2$ condition \eqref{eq:newA2} is of course independent of the interpretation of the Hilbert transform.

\begin{remark}[(Disjoint vs.\ separated intervals)]\label{rem:disjVsSep}
Having the identity \eqref{eq:Bdisjoint} for all disjointly (rather than just separatedly) supported pairs of $f,g\in\mathscr{F}$ is necessary for the existence of the Hilbert transform, in the following sense:

Suppose that we have a bilinear form $B:\mathscr{F}\times\mathscr{F}\to\R$ that is only assumed to satisfy \eqref{eq:Bdisjoint} for $f,g\in\mathscr{F}$ with $\dist(\spt f,\spt g)>0$. If $B$ satisfies the global testing conditions \eqref{eq:BtestGlobal} (in particular, if $B$ is bounded in the sense of \eqref{eq:Bbounded}), then \eqref{eq:Bdisjoint} also holds for any $f,g\in\mathscr{F}$ with disjoint supporting intervals.
\end{remark}

\begin{proof}
By linearity and symmetry, if suffices to prove \eqref{eq:Bdisjoint} for $f=1_I$, $g=1_J$, with $I=[a,b)$ and $J=[b,c)$. Let $b_n\uparrow b$ be dyadic points, and $I_n:=[a,b_n)$. Then \eqref{eq:Bdisjoint} holds for $(f,g)=(1_{I_n},1_J)$. As $n\to\infty$, we have $1_{I_n}\to 1_I$ in $L^2(\sigma)$. Since \eqref{eq:BtestGlobal} says that $f\mapsto B(f,1_I)$ is continuous in the norm of $L^2(\sigma)$, it follows that $B(1_{I_n},1_J)\to B(1_I,1_J)$. On the other hand, the kernel $1/(x-y)$ is positive for $(x,y)\in J\times I$, and hence
\begin{equation*}
  \iint\frac{1}{x-y}1_{I_n}(y)\ud\sigma(y)1_J(x)\ud w(x)\to   \iint\frac{1}{x-y}1_{I}(y)\ud\sigma(y)1_J(x)\ud w(x)
\end{equation*}
by the monotone convergence theorem.
\end{proof}

\subsection{Necessity of the $A_2$ condition}

Recall that Theorem~\ref{thm:main} asserts the equivalence of three different conditions, \eqref{eq:Hbounded} $\Leftrightarrow$ \eqref{eq:HtestGlobal} $\Leftrightarrow$ (\eqref{eq:HtestLocal} and \eqref{eq:newA2}), where \eqref{eq:newA2} is our new $A_2$ condition. It is obvious that \eqref{eq:Hbounded} $\Rightarrow$ \eqref{eq:HtestGlobal} $\Rightarrow$ \eqref{eq:HtestLocal}. The main implication of the theorem consists of showing that (\eqref{eq:HtestLocal} and \eqref{eq:newA2}) $\Rightarrow$ \eqref{eq:Hbounded}. We now clear the straightforward but not completely trivial implication \eqref{eq:HtestGlobal} $\Rightarrow$ \eqref{eq:newA2} out of the way.

In fact, let us formulate the intermediate condition
\begin{equation}\label{eq:HtestOff}
  \Norm{1_{I^c}H(1_I\ud\sigma)}{L^2(w)}\leq \mathcal{H}_{\operatorname{off}}\sigma(I)^{1/2},
\end{equation}
where clearly $\mathcal{H}_{\operatorname{off}}\leq\mathcal{H}_{\operatorname{glob}}$. We will now show that

\begin{lemma}\label{lem:A2nec}
\begin{equation*}
  [\sigma,w]_{A_2}^*\leq 2\mathcal{H}_{\operatorname{off}}.
\end{equation*}
\end{lemma}

\begin{proof}
We dualize \eqref{eq:HtestOff} with a function $g\in L^2(w)$, which is positive on the right of $I$ and negative on its left. Let also $E\subset I^c$ be a bounded set strictly separated from $I$. Then \eqref{eq:HtestOff} gives
\begin{equation*}
  \Babs{\iint_{E\times I}\frac{1}{x-y}\ud\sigma(y)g(x)\ud w(x)}\leq\mathcal{H}_{\operatorname{off}}\sigma(I)^{1/2}\Norm{1_E g}{L^2(w)}.
\end{equation*}
The key point is that $g(x)/(x-y)$ has constant sign throughout the domain of integration, so we may equivalently take absolute value inside the integral. The next observation is that $\abs{x-y}\leq 2\abs{x-c_I}$ for $(x,y)\in I^c\times I$. Hence, in fact,
\begin{equation*}
  \int_I\ud\sigma(y)\int_E \frac{\abs{g(x)}}{\abs{x-c_I}}\ud w(x)\leq 2\mathcal{H}_{\operatorname{off}}\sigma(I)^{1/2}\Norm{1_E g}{L^2(w)}.
\end{equation*}
Choosing $g(x)=\sign(x-c_I)\abs{x-c_I}^{-1}$, we further deduce that
\begin{equation*}
  \sigma(I)\int_E\frac{1}{(x-c_I)^2}\ud w(x)\leq 2\mathcal{H}_{\operatorname{off}}\sigma(I)^{1/2}\Big(\int_E\frac{1}{(x-c_I)^2}\ud w(x)\Big)^{1/2}.
\end{equation*}
Both $\sigma(I)$ and $\int_E(x-c_I)^{-2}\ud w(x)$ are finite for Radon measures $\sigma$ and $w$. If they are nonzero, we may divide to get
\begin{equation*}
  \sigma(I)^{1/2}\Big(\int_E\frac{1}{(x-c_I)^2}\ud w(x)\Big)^{1/2}\leq 2\mathcal{H}_{\operatorname{off}},
\end{equation*}
and this is obvious if either factor is zero. Letting $E$ approach the complement of $I$, we deduce that
\begin{equation*}
  [\sigma,w]_{A_2}^*\leq 2\mathcal{H}_{\operatorname{off}}.
\end{equation*}
(Possible issues with point masses at the end-points of $I$ may be easily resolved. For instance, if $I=[a,b)$ and $w(\{b\})>0$, one might first derive the estimate for $[a,b-\epsilon)$ in place of $I$, in which case $[b,\infty)$ is still strictly separated, and then let $\epsilon\to 0$.) 
\end{proof}

Although our characterization of the Hilbert transform inequality requires the Muckenhoupt--Poisson $A_2$ condition~\eqref{eq:newA2}, there are many parts of the argument that are conveniently handled with the simpler, symmetric $A_2$ constant
\begin{equation}\label{eq:simpleA2}
  [\sigma,w]_{A_2}:=\sup\frac{\sigma(I)^{1/2}w(J)^{1/2}}{\abs{I}},
\end{equation}
where the supremum is over all pairs of adjacent intervals $I,J$ of equal length. This satisfies:

\begin{lemma}
\begin{equation*}
  [\sigma,w]_{A_2}\leq\frac{3}{2}([\sigma,w]_{A_2}^*\wedge[w,\sigma]_{A_2}^*).
\end{equation*}
\end{lemma}

\begin{proof}
If $x\in J$, then $\abs{x-c_I}\leq\frac32\abs{I}$, and hence
\begin{equation*}
  \sigma(I)\int_{I^c}\frac{\ud w(x)}{(x-c_I)^2}
  \geq \sigma(I)\int_{J}\frac{\ud w(x)}{(\frac32\abs{I})^2}=\frac{\sigma(I)w(J)}{(\frac32\abs{I})^2}.
\end{equation*}
The other upper bound follows by symmetry.
\end{proof}

\subsection{The weak boundedness property}

It is not uncommon in different $T(1)$ theorems to assume a weak boundedness property of the type
\begin{equation}\label{eq:WBP}
  \abs{B(1_I,1_J)}\leq\mathcal{W}\sigma(I)^{1/2}w(J)^{1/2},
\end{equation}
which is seen to be nothing other than the restricted boundedness of the bilinear form, when both its arguments are indicators of intervals. In the two-weight context, the role of such a condition has been identified in \cite[Theorem~2.3(1)]{NTV:2wHaar}, and in the more specific context of the Hilbert transform, in \cite[Eq.~(2.8)]{LSU:2wHilbertEnergy}. In our present setting, we will check that \eqref{eq:WBP} is a consequence of the assumptions \eqref{eq:HtestLocal} and \eqref{eq:newA2}. This verification involves somewhat different reasoning for different relative position of the intervals $I$ and $J$.

\begin{lemma}\label{lem:IsubJ}
If $I\subseteq J$, then
\begin{equation*}
  \abs{B(1_I,1_J)}\leq\mathcal{H}^*\sigma(I)^{1/2}w(J)^{1/2}.
\end{equation*}
\end{lemma}

\begin{proof}
This is immediate from \eqref{eq:HtestLocal}, or rather its bilinear formulation \eqref{eq:BtestLocal}, applied to $f=1_I$.
\end{proof}

\begin{lemma}\label{lem:Isub3JminusJ}
Suppose that $I\subset 3J\setminus J$. Then
\begin{equation*}
  \abs{B(1_I,1_J)}\leq 2[\sigma,w]_{A_2}\sigma(I)^{1/2}w(I)^{1/2}.
\end{equation*}
\end{lemma}

\begin{proof}
Let $K$ be the interval of length $\abs{J}$ and adjacent to $J$ that contains $I$. Then
\begin{equation*}
  \abs{B(1_I,1_J)}
  =\Babs{\iint_{I\times J}\frac{\ud\sigma(y)\ud w(x)}{x-y}}
  \leq 
  \iint_{K\times J}\frac{\ud\sigma|_I(y)\ud w(x)}{\abs{x-y}},
\end{equation*}
where $\sigma|_I(E):=\sigma(I\cap E)$. Let $\mu:=\sigma|_I\times w$ be the product measure. (Notice that, if $I$ and $J$ are adjacent, the equality above makes use of the postulate that \eqref{eq:Bdisjoint} is valid for all disjointly, and not just separatedly, supported functions.)

Fix a small dyadic rational $\eta>0$. Let us denote $K_0:=K$, and inductively $K_{k+1}\subset K_k$ is the subinterval of length $(1-\eta)\abs{K_k}$ and adjacent to $J$; similarly $J_0:=J$ and $J_{k+1}\subset J_k$ is the subinterval of length $(1-\eta)\abs{J_k}$ and adjacent to $K$. Let $Q:=K\times J$ and $Q_k:=K_k\times J_k$. Then
\begin{equation*}
  Q=\bigcup_{k=0}^\infty Q_k\setminus Q_{k+1}
\end{equation*}
and
\begin{equation*}
\begin{split}
  \iint_{K\times J}\frac{\ud\sigma|_I(y)\ud w(x)}{\abs{x-y}}
  &=\sum_{k=0}^\infty\iint_{Q_k\setminus Q_{k+1}}\frac{\ud\mu(y,x)}{\abs{x-y}} 
   \leq \sum_{k=0}^\infty\iint_{Q_k\setminus Q_{k+1}}\frac{\ud\mu(y,x)}{\abs{K_{k+1}}} \\
   &=\sum_{k=0}^\infty\frac{1}{\abs{K_{k+1}}}(\mu(Q_k)-\mu(Q_{k+1})) \\
   &=\sum_{k=0}^\infty\frac{\sqrt{\mu(Q_k)}+\sqrt{\mu(Q_{k+1})}}{\abs{K_{k+1}}}(\sqrt{\mu(Q_k)}-\sqrt{\mu(Q_{k+1})}) \\
   &\overset{(*)}{\leq} [\sigma,w]_{A_2}\sum_{k=0}^\infty\frac{\abs{K_k}+\abs{K_{k+1}}}{\abs{K_{k+1}}}(\sqrt{\mu(Q_k)}-\sqrt{\mu(Q_{k+1})}) \\
   &=\Big(\frac{1}{1-\eta}+1\Big)[\sigma,w]_{A_2}\sum_{k=0}^\infty (\sqrt{\mu(Q_k)}-\sqrt{\mu(Q_{k+1})})  \\
    &=\frac{2-\eta}{1-\eta}[\sigma,w]_{A_2}\sqrt{\mu(Q_0)}=\frac{2-\eta}{1-\eta}[\sigma,w]_{A_2}\sqrt{\sigma(I)w(J)},
\end{split}
\end{equation*}
where $(*)$ was based on
\begin{equation*}
  \sqrt{\mu(Q_k)}
  =\sqrt{\sigma(I\cap K_k)w(J_k)}
  \leq \sqrt{\sigma(K_k)w(J_k)}\leq[\sigma,w]_{A_2}\abs{K_k}.
\end{equation*}
The claim follows by taking the limit $\eta\to 0$.
\end{proof}

\begin{lemma}\label{lem:Isub3Jc}
If $I\subseteq(3J)^c$, then
\begin{equation*}
  \abs{B(1_I,1_J)}\leq\frac32[w,\sigma]_{A_2}^*\sigma(I)^{1/2}w(J)^{1/2}.
\end{equation*}
\end{lemma}

\begin{proof}
For $x\in J$ and $y\in I\in (3J)^c$, we have $\abs{x-y}\geq\frac23\abs{y-c_J}$. Thus
\begin{equation*}
\begin{split}
  \abs{B(1_I,1_J)}
  &\leq\iint_{I\times J}\frac{\ud\sigma(y)\ud w(x)}{\abs{x-y}}
    \leq\frac32\int_I\frac{\ud\sigma(y)}{\abs{y-c_J}}w(J) \\
  &\leq\frac32\sigma(I)^{1/2}\Big(\int_{J^c}\frac{\ud\sigma(y)}{(y-c_J)^2}\Big)^{1/2}w(J)
  \leq\frac32[w,\sigma]_{A_2}^*\sigma(I)^{1/2}w(J)^{1/2}.
\end{split}
\end{equation*}
\end{proof}

\begin{proposition}
The conditions \eqref{eq:HtestLocal} and \eqref{eq:newA2} imply the weak boundedness property \eqref{eq:WBP}; more precisely,
\begin{equation*}
  \mathcal{W}\lesssim(\mathcal{H}+[\sigma,w]_{A_2}^*)\wedge(\mathcal{H}^*+[w,\sigma]_{A_2}^*).
\end{equation*}
\end{proposition}

\begin{proof}
For any intervals $I$ and $J$, we have
\begin{equation*}
  B(1_I,1_J)
  =B(1_{I\setminus(3J)},1_J)+B(1_{I\cap(3J\setminus J)},1_J)+B(1_{I\cap J},1_J),
\end{equation*}
where $I\setminus(3J)$ is a union of at most two intervals in the complement of $3J$ (where Lemma~\ref{lem:Isub3Jc} applies), $I\cap(3J\setminus J)$ is a union of at most two intervals contained in $3J\setminus J$ (where Lemma~\ref{lem:Isub3JminusJ} applies), and of course $I\cap J$ (if non-empty) is an interval contained in $J$ (where Lemma~\ref{lem:IsubJ} applies). Combinining the bounds from these Lemmas, and recalling that $[\sigma,w]_{A_2}\lesssim[w,\sigma]_{A_2}^*$, gives the second of the two claimed upper bounds. The first one follows by symmetry, by splitting $J$ instead of $I$.
\end{proof}


\subsection{Truncated Hilbert transforms}\label{subsec:truncated}
We include this section to facilitate the comparison of our results with those of Lacey, Sawyer, Shen, and Uriarte-Tuero \cite{LSSU:2wHilbert}, whose approach goes via the truncated Hilbert transforms
\begin{equation*}
  H_\epsilon(f\ud\sigma)(x):=\int_{\abs{x-y}>\epsilon}\frac{1}{x-y}f(y)\ud\sigma(y).
\end{equation*}
The pairing
\begin{equation*}
  B_\epsilon(f,g):=\pair{H_\epsilon(f\ud\sigma)}{g\ud w}
  =\iint_{\abs{x-y}>\epsilon}\frac{1}{x-y}f(y)\ud\sigma(y)g(x)\ud w(x)
\end{equation*}
is well defined for all $f,g\in\mathscr{F}$.

\begin{lemma}\label{lem:BepsCadlag}
For $f,g\in\mathscr{F}$, the function $\epsilon\in(0,\infty)\mapsto B_\epsilon(f,g)$ is right-continuous with left limits (``c\`adl\`ag'').
\end{lemma}

\begin{proof}
The sets $\{\abs{x-y}>\eta\}$ converge to $\{\abs{x-y}>\epsilon\}$ as $\eta\downarrow\epsilon$, and to $\{\abs{x-y}\geq\epsilon\}$ as $\eta\uparrow\epsilon$.
By dominated convergence, this implies that $\lim_{\eta\downarrow\epsilon}B_\eta(f,g)=B_\epsilon(f,g)$ and
\begin{equation*}
  \lim_{\eta\uparrow\epsilon}B_\eta(f,g)=\iint_{\abs{x-y}\geq\epsilon}\frac{1}{x-y}f(y)\ud\sigma(y)g(x)\ud w(x).\qedhere
\end{equation*}
\end{proof}

Suppose, moreover, that
\begin{equation}\label{eq:veryWBP}
  \sup_{\epsilon>0}\abs{B_\epsilon(1_I,1_I)}<\infty
\end{equation}
for all intervals $I$ with dyadic end-points. This follows in particular from either of the truncated testing conditions of Lacey et al.~\cite{LSSU:2wHilbert},
\begin{equation}\label{eq:LSSUtesting}
\begin{split}
  \sup_{\epsilon>0}\Norm{1_I H_\epsilon(1_I\ud\sigma)}{L^2(w)} &\leq C\sigma(I)^{1/2}, \\
  \sup_{\epsilon>0}\Norm{1_I H_\epsilon(1_I\ud w)}{L^2(\sigma)} &\leq Cw(I)^{1/2}.
\end{split}
\end{equation}

We also assume the $A_2$ conditions \eqref{eq:newA2}.

\begin{lemma}\label{lem:BepsBounded}
For $f,g\in\mathscr{F}$, the function $\epsilon\in(0,\infty)\mapsto B_\epsilon(f,g)$ is bounded.
\end{lemma}

\begin{proof}
By linearity, it suffices to consider $f=1_I$ and $g=1_J$, where $I$ and $J$ are intervals with dyadic end-points. Then
\begin{equation*}
  B_\epsilon(1_I,1_J)
  =B_\epsilon(1_{I\setminus J},1_J)+B_\epsilon(1_{I\cap J},1_{I\cap J})+B_\epsilon(1_{I\cap J},1_{J\setminus I}).
\end{equation*}
The middle term is uniformly (in $\epsilon>0$) bounded by the assumption \eqref{eq:veryWBP}. The other two terms are sums of at most two expressions $B_\epsilon(1_K,1_L)$ with disjoint intervals $K,L$, and these satisfy
\begin{equation*}
  \abs{B_\epsilon(1_K,1_L)}
  \leq\iint_{K\times L}\frac{1}{\abs{x-y}}\ud\sigma(y)\ud w(x)
  \lesssim ([\sigma,w]_{A_2}^*+[w,\sigma]_{A_2}^*)\sigma(K)^{1/2}w(L)^{1/2}
\end{equation*}
where the last step is a consequence of Lemmas~\ref{lem:Isub3JminusJ} and \ref{lem:Isub3Jc}.
\end{proof}

\begin{lemma}\label{lem:limitAtZero}
If $f,g\in\mathscr{F}$ are disjointly supported, then
\begin{equation}\label{eq:limitAtZero}
  \lim_{\epsilon\to 0}B_\epsilon(f,g)=\iint\frac{1}{x-y}f(y)\ud\sigma(y)g(x)\ud w(x).
\end{equation}
\end{lemma}

\begin{proof}
By linearity, it suffices to consider $f=1_I$, $g=1_J$ with $I\cap J=\varnothing$. The finiteness of the right hand side follows from Lemmas~\ref{lem:Isub3JminusJ} and \ref{lem:Isub3Jc}. Now
\begin{equation*}
  B_\epsilon(1_I,1_J)=\iint_{\substack{(y,x)\in I\times J\\ \abs{x-y}>\epsilon}}\frac{1}{x-y}\ud\sigma(y)\ud\sigma(x).
\end{equation*}
Since $\abs{y-x}>0$ for every $(y,x)\in I\times J$, the integration domain converges to $I\times J$, and thus the asserted limit follows from dominated (or monotone) convergence.
\end{proof}

\begin{proposition}\label{prop:LSSUisSpecialCase}
Let $\sigma$ and $w$ be Radon measures that satisfy the $A_2$ conditions \eqref{eq:newA2} and the finiteness condition \eqref{eq:veryWBP}. Then there exists a bilinear form $B$ on $\mathscr{F}\times\mathscr{F}$ that satisfies \eqref{eq:Bdisjoint} for all disjointly supported $f,g\in\mathscr{F}$.

If, in addition, the measures satisfy the testing conditions \eqref{eq:LSSUtesting}, then the bilinear form $B$ satisfies the local testing conditions \eqref{eq:BtestLocal}.
\end{proposition}

\begin{proof}
By Lemmas~\ref{lem:BepsCadlag} and \ref{lem:BepsBounded}, we have
\begin{equation*}
  [\epsilon\mapsto B_\epsilon(f,g)]\in D_b(0,\infty):=\{\phi:(0,\infty)\to\R\text{ is c\`adl\`ag and bounded}\}
\end{equation*}
for all $f,g\in\mathscr{F}$, and by Lemma~\ref{lem:limitAtZero},
\begin{equation*}
  [\epsilon\mapsto B_\epsilon(f,g)]\in D_{\ell}(0,\infty):=\{\phi\in D_b(0,\infty),\lim_{t\to 0}\phi(t)\text{ exists}\}
\end{equation*}
for disjointly supported $f,g\in\mathscr{F}$.

Both equipped with the supremum norm, $D_\ell(0,\infty)$ is a closed subspace of the Banach space $D_b(0,\infty)$, and $\phi\mapsto\lim_{t\to 0}\phi(t)$ is a continuous linear functional on $D_\ell(0,\infty)$, of norm $1$. By Hahn--Banach, it has an extension $\Lambda$ to a continuous linear functional defined on all of $D_b(0,\infty)$ and of the same norm. So we may now abstractly define
\begin{equation*}
   B(f,g):=\Lambda(\epsilon\mapsto B_{\epsilon}(f,g)).
\end{equation*}
This is a bilinear form by the bilinearity of $B_\epsilon$ and the linearity of $\Lambda$, and it has the integral representation \eqref{eq:Bdisjoint} for disjointly supported $f,g\in\mathscr{F}$ by Lemma~\ref{lem:limitAtZero}.

To verify the testing conditions under the hypothesis \eqref{eq:LSSUtesting}, we compute
\begin{equation*}
\begin{split}
  \abs{B(1_I,1_I g)}
  &=\abs{\Lambda(\epsilon\mapsto B_\epsilon(1_I,1_Ig))}
  \leq\sup_{\epsilon>0}\abs{B_\epsilon(1_I,1_Ig)} \\
  &\leq\sup_{\epsilon>0}\Norm{1_I H_\epsilon(1_I\ud\sigma)}{L^2(w)}\Norm{g}{L^2(w)}
  \leq C\sigma(I)^{1/2}\Norm{g}{L^2(w)},
\end{split}
\end{equation*}
and the other condition in \eqref{eq:BtestLocal} is shown in the same way.
\end{proof}

Proposition~\ref{prop:LSSUisSpecialCase} shows in particular that the assumptions of Lacey et al.~\cite{LSSU:2wHilbert} imply the local assumptions \eqref{eq:HtestLocal} and \eqref{eq:newA2} of our Theorem~\ref{thm:main} which, once proven, implies the boundedness of the two-weight Hilbert transform in the sense described above. On the other hand, Lacey et al.~\cite{LSSU:2wHilbert} obtain the apparently stronger conclusion (in the special case of no common point masses) of the uniform boundedness of the truncated Hilbert transforms. In the following Proposition, however, we show that this boundedness of the truncated operators is a relatively simple consequence of the boundedness of the un-truncated Hilbert transform.

\begin{proposition}
Let $\sigma$ and $w$ be two Radon measures that satisfy the $A_2$ condition \eqref{eq:simpleA2}.
Suppose that there exists a bounded bilinear form $B:L^2(\sigma)\times L^2(w)\to\R$ such that
\begin{equation*}
  B(f,g)=\iint\frac{1}{x-y}f(y)g(x)\ud\sigma(y)\ud w(x),
\end{equation*}
whenever $f,g\in\mathscr{F}$ satisfy $\dist(\spt f,\spt g)>0$.
Then the truncated Hilbert transforms satisfy the uniform estimate
\begin{equation*}
  \sup_{\epsilon>0}\Norm{H_{\epsilon}(f\ud\sigma)}{L^2(w)}\leq C\Norm{f}{L^2(\sigma)}.
\end{equation*}
\end{proposition}

\begin{proof}
We proceed in two steps, and first consider a dyadic version of the truncation,
\begin{equation*}
  H^d_{k}(f\ud\sigma):=\sum_{I\in\mathscr{D}_k}1_I H(1_{(2I)^c}f\ud\sigma)
  =H(f\ud\sigma)-\sum_{I\in\mathscr{D}_k}1_I H(1_{2I} f\ud\sigma),
\end{equation*}
where $\mathscr{D}_k:=\{I\in\mathscr{D}:\abs{I}=2^{-k}\}$ and $\mathscr{D}$ is a fixed dyadic system.

The first term is bounded by the assumed boundedness of $H$. For the second part, we have
\begin{equation*}
\begin{split}
  \BNorm{\sum_{I\in\mathscr{D}_k}1_I H(1_{2I} f\ud\sigma)}{L^2(w)}
  &\leq\Big(\sum_{I\in\mathscr{D}_k}\Norm{H(1_{2I} f\ud\sigma)}{L^2(w)}^2\Big)^{1/2} \\
  &\leq N\Big(\sum_{I\in\mathscr{D}_k}\Norm{1_{2I} f}{L^2(\sigma)}^2\Big)^{1/2}
  \leq \sqrt{2}N\Norm{f}{L^2(\sigma)},
\end{split}
\end{equation*}
by the bounded overlap of the intervals $2I$ for $I\in\mathscr{D}_k$. Thus $H^d_k$ is bounded by $(1+\sqrt{2})N$, where $N$ is the norm bound for $H$.

In the second step, we estimate the difference of $H_{\epsilon}$ and $H^d_k$. Note that
\begin{equation*}
  H^d_k(f\ud\sigma)(x)=\int_{J^c}\frac{1}{x-y}f(y)\ud\sigma
\end{equation*}
for some interval $J=[x-a,x+b)$, where $\frac12 2^{-k}\leq a,b\leq \frac32 2^{-k}$. Let us choose $k$ so that $\frac14 2^{-k}\leq \epsilon<\frac12 2^{-k}$, so that necessarily $[x-\epsilon,x+\epsilon]\subset J$, and
\begin{equation*}
\begin{split}
  \abs{H_\epsilon(f\ud\sigma)(x)-H^d_k(f\ud\sigma)(x)}
  &=\Babs{\int_{J\setminus[x-\epsilon,x+\epsilon]}\frac{1}{x-y}f(y)\ud\sigma(y)} \\
  &\leq\frac{1}{\epsilon}\int_{\epsilon<\abs{x-y}\leq 6\epsilon}\abs{f(y)}\ud\sigma(y) \\
  &\lesssim\sum_{0<\abs{m}\leq 12} \frac{1}{\abs{I}}\int_{I\dot+m}\abs{f}\ud\sigma,\qquad
  I\dot+m:=I+\ell(I)m,
\end{split}
\end{equation*}
where $I$ is the unique dyadic interval of length $2^{-k-3}\in(\frac12\epsilon,\epsilon]$ that contains $x$.

Thus, we have
\begin{equation*}
  \Norm{H_{\epsilon}(f\ud\sigma)-H^d_k(f\ud\sigma)}{L^2(w)}
  \lesssim\BNorm{\sum_{I\in\mathscr{D}_{k+3}}\sum_{0<\abs{m}\leq 12} \frac{1_I}{\abs{I}}\int_{I\dot+m}\abs{f}\ud\sigma }{L^2(w)}.
\end{equation*}
We dualize with a non-negative $g\in L^2(w)$ and observe that
\begin{equation*} 
\begin{split}
  \sum_{I\in\mathscr{D}_{k+3}} &\frac{1}{\abs{I}}\int_{I\dot+m}\abs{f}\ud\sigma\int_I g\ud w
  \leq\sum_{I\in\mathscr{D}_{k+3}}\frac{\sqrt{\sigma(I\dot+m)w(I)}}{\abs{I}}\Norm{1_{I\dot+m}f}{L^2(\sigma)}\Norm{1_I g}{L^2(w)} \\
  &\lesssim [\sigma,w]_{A_2}\sum_{I\in\mathscr{D}_{k+3}}\Norm{1_{I\dot+m}f}{L^2(\sigma)}\Norm{1_I g}{L^2(w)}
  \leq[\sigma,w]_{A_2}\Norm{f}{L^2(\sigma)}\Norm{g}{L^2(w)},
\end{split}
\end{equation*}
where the last step follows from Cauchy--Schwarz. It remains to sum over the boundedly many values of $m$ to complete the estimate.
\end{proof}

\begin{remark}[(Another construction of the Hilbert transform)]\label{rem:anotherConstr}
Due to the complex-analytic connections of the Hilbert transform, it it is also natural to consider
\begin{equation*}
\begin{split}
  H^t(f\ud\sigma)(x) &:=\int_{\R}\frac{1}{x+it-y}f(y)\ud\sigma(y),\\
  B^t(f,g) &:=\pair{H^t(f\ud\sigma)}{g\ud w}=\iint_{\R\times \R}\frac{1}{x+it-y}f(y)\ud\sigma(y)g(x)\ud w(x),
\end{split}
\end{equation*}
which, like $H_\epsilon$ and $B_\epsilon$, are well defined for all $f,g\in\mathscr F$, when $t>0$. It is immediate that $t\in(0,\infty)\mapsto B^t(f,g)$ is continuous for $f,g\in\mathscr F$ and, as in Lemma \ref{lem:limitAtZero}, one checks that $B^t(f,g)\to \iint\frac{1}{x-y}f(y)\ud\sigma(y)g(x)\ud w(x)$ as $t\to 0$ if $f,g\in\mathscr F$ are disjointly supported. Assuming the analogue of \eqref{eq:veryWBP}, namely that for all intervals $I$ with dyadic end-points we have
\begin{equation*}
  \sup_{t>0}\abs{B^t(1_I,1_I)}<\infty,
\end{equation*}
we deduce the analogue of Lemma \ref{lem:BepsBounded}, the boundedness of $t\in(0,\infty)\mapsto B^t(f,g)$ for all $f,g\in\mathscr F$. Denoting by $C_b(0,\infty)$ the Banach space of all continuous, bounded functions on $(0,\infty)$ with the supremum norm, and by $C_\ell(0,\infty)$ its subspace of functions with limit as $t\to 0$, we obtain, similarly to Proposition \ref{prop:LSSUisSpecialCase},  an alternative construction of a bilinear form of the Hilbert transform as
\begin{equation*}
  \tilde B(f,g):=\tilde\Lambda(t\mapsto B^t(f,g)),
\end{equation*}
where $\tilde\Lambda\in(C_b(0,\infty))^*$ is any Hahn--Banach extension of the continuous linear functional $\phi\mapsto\lim_{t\to 0}\phi(t)$ on $C_\ell(0,\infty)$. Note that, while the comparison of the complex-analytic $B^\epsilon(f,g)$ and the truncated $B_\epsilon(f,g)$ is completely routine in the classical (say, unweighted) theory, it does not seem as immediate in the generality of the two-weight inequalities that we consider.
\end{remark}

%
%

\subsection{The probabilistic reduction}

We recall the martingale difference expansions and the probabilistic good/bad decompositions, the established framework of non-homogeneous analysis that essentially goes back to \cite{NTV:Tb}. We do this in detail in order to make it clear that we do not impose any \emph{a priori} boundedness assumptions beyond the testing conditions \eqref{eq:BtestLocal}.

Every $f\in\mathscr{F}$ is piecewise constant on some sufficiently fine dyadic partition of $\R$, say on the dyadic intervals of length $2^{-m}$. So it suffices to bound $B(f,g)$ for all $f,g\in\mathscr{F}$ of this type, where $m$ is arbitrary but fixed, as long as the bound is independent of $m$. Thus we consider only functions of this form from now on, and only intervals with dyadic end-points on the scale $2^{-m}$. We refer to the dyadic intervals of length $2^{-m}$ as the \emph{minimal intervals}.

Besides the original dyadic system $\mathscr{D}^0$, we will be making use of shifted systems $\mathscr{D}$, obtained by translating the intervals of $\mathscr{D}^0$ by a multiple of the minimal length scale $2^{-m}$. Thus the minimal intervals of the different dyadic systems always coincide.

Consider a function $f\in\mathscr{F}$, supported by some $I_0\in\mathscr{D}$, where $\mathscr{D}$ is a shifted system as described. Then we have the familiar martingale difference expansion
\begin{equation}\label{eq:mdsExpansion}
  f=E_{I_0}^\sigma f+\sum_{\substack{I\in\mathscr{D} \\ I\subseteq I_0}}\Delta_I^\sigma f, 
\end{equation}
where, writing
\begin{equation*}
  \ave{f}_{I}^\sigma :=\fint_I f\ud\sigma:=\frac{1}{\sigma(I)}\int_I f\ud\sigma, 
\end{equation*}
and, denoting by $I_{\operatorname{left}}$ and $I_{\operatorname{right}}$ the left and right halves (the dyadic children) of $I$, we define
\begin{equation*}
  E_I^\sigma f:= \ave{f}_{I}^\sigma 1_{I},\qquad
    \Delta_I^\sigma f:=\sum_{u\in\{\operatorname{left},\operatorname{right}\}}E_{I_u}^\sigma f-E_I^\sigma f.
\end{equation*}
Only the intervals $I$ strictly larger than the minimal intervals need to be taken into account in the sum in \eqref{eq:mdsExpansion}, since $\Delta_I^\sigma f=0$ if $I$ is a subset of a minimal interval. The martingale differences $\Delta_I^w$ are of course defined analogously.

We record the following useful bounds:

\begin{lemma}\label{lem:basicBound}
For all $f,g\in\mathscr{F}$, we have
\begin{equation*}
\begin{split}
  \abs{B(\Delta_I^\sigma f, \Delta_J^w g)} &\leq 2\mathcal{W}\Norm{\Delta_I^\sigma f}{L^2(\sigma)}\Norm{\Delta_J^ w g}{L^2(w)}, \\
  \abs{B(E_I^\sigma f, \Delta_J^w g)} &\leq \sqrt{2}\mathcal{W}\Norm{E_I^\sigma f}{L^2(\sigma)}\Norm{\Delta_J^ w g}{L^2(w)}, \\
  \abs{B(\Delta_I^\sigma f, E_J^w g)} &\leq \sqrt{2}\mathcal{W}\Norm{\Delta_I^\sigma f}{L^2(\sigma)}\Norm{E_J^ w g}{L^2(w)}, \\
  \abs{B(E_I^\sigma f, E_J^w g)} &\leq \mathcal{W}\Norm{E_I^\sigma f}{L^2(\sigma)}\Norm{E_J^ w g}{L^2(w)}.
\end{split}
\end{equation*}
\end{lemma}

\begin{proof}
We have
\begin{equation*}
\begin{split}
  \abs{B(\Delta_I^\sigma f,\Delta_J^w g)}
  &\leq\sum_{u,v\in\{\operatorname{left},\operatorname{right}\}}\abs{\ave{\Delta_I^\sigma f}_{I_u}^\sigma}\cdot
      \abs{\ave{\Delta_I^w g}_{J_v}^\sigma}\cdot\abs{ B(1_{I_u},1_{J_v})} \\
  &\leq\sum_{u,v\in\{\operatorname{left},\operatorname{right}\}}
      \abs{\ave{\Delta_I^\sigma f}_{I_u}^\sigma}\cdot\abs{\ave{\Delta_J^w g}_{J_v}^\sigma}\cdot
    \mathcal{W} \sigma(I_u)^{1/2}w(J_v)^{1/2} \\
   &\leq\mathcal{W}\cdot
   \sqrt{4}\Big(\sum_{u,v\in\{\operatorname{left},\operatorname{right}\}}\abs{\ave{\Delta_I^\sigma f}_{I_u}^\sigma}^2\sigma(I_u)
       \abs{\ave{\Delta_J^w g}_{J_v}^\sigma}^2 w(J_v)\Big)^{1/2} \\
    &=2\mathcal{W}\Norm{\Delta_I^\sigma f}{L^2(\sigma)}\Norm{\Delta_J^w g}{L^2(w)}.
\end{split}
\end{equation*}
The proofs of the other estimates are similar.
\end{proof}

Let $\mathcal{K}_n$ be the supremum of $\abs{B(f,g)}$ taken over all $f,g\in\mathscr{F}$ with $\Norm{f}{L^2(\sigma)}\leq 1$ and $\Norm{g}{L^2(w)}\leq 1$, and for which there exist at most $2^n$ consecutive minimal intervals supporting both $f$ and $g$. To prove the boundedness of the Hilbert transform, it suffices to bound the numbers $\mathscr{K}_n$ uniformly in $n$. The following qualitative \emph{a priori} bound is useful:

\begin{lemma}\label{lem:aPrioriBound}
Under the weak boundedness property~\eqref{eq:WBP}, we have
\begin{equation*}
  \mathcal{K}_n<\infty\qquad\forall\ n\in\N.
\end{equation*}
\end{lemma}

\begin{proof}
Let $I_0$ be the union of $2^n$ consecutive minimal intervals supporting both $f$ and $g$, and $\mathscr{D}$ be a shifted dyadic system with $I_0\in\mathscr{D}$. We make use of the martingale difference expansion \eqref{eq:mdsExpansion} and Lemma~\ref{lem:basicBound} to estimate
\begin{equation*}
\begin{split}
  \abs{B(f,g)}
  &\leq\abs{B(E_{I_0}^\sigma f,E_{I_0}^w g)}
     +\sum_{I:I\subseteq I_0}\abs{B(\Delta_{I}^\sigma f,E_{I_0}^w g)} \\
  &\qquad   +\sum_{J:J\subseteq I_0}\abs{B(E_{I_0}^\sigma f,\Delta_J^w g)}
     +\sum_{I,J:I,J\subseteq I_0}\abs{B(\Delta_{I}^\sigma f,\Delta_J^w g)} \\
   &\leq 2\mathcal{W}\Big(\Norm{E_{I_0}^\sigma f}{L^2(\sigma)}\Norm{E_{I_0}^w g}{L^2(w)}
     +\sum_{I:I\subseteq I_0}\Norm{\Delta_{I}^\sigma f}{L^2(\sigma)}\Norm{E_{I_0}^w g}{L^2(w)} \\
  &\qquad   +\sum_{J:J\subseteq I_0}\Norm{E_{I_0}^\sigma f}{L^2(\sigma)}\Norm{\Delta_{J}^w g}{L^2(w)}
     +\sum_{I,J:I,J\subseteq I_0}\Norm{\Delta_{I}^\sigma f}{L^2(\sigma)}\Norm{\Delta_{J}^w g}{L^2(w)}\Big)  \\
   &=2\mathcal{W}\Big(\Norm{E_{I_0}^\sigma f}{L^2(\sigma)}+\sum_{I:I\subseteq I_0}\Norm{\Delta_{I}^\sigma f}{L^2(\sigma)}\Big) \\
   &\qquad\times   \Big(\Norm{E_{I_0}^w g}{L^2(w)}+\sum_{I:I\subseteq I_0}\Norm{\Delta_{I}^w g}{L^2(w)}\Big).
\end{split}
\end{equation*}
Since $\Delta_I^\sigma f=0$ for $\abs{I}\leq 2^{-m}=2^{-n}\abs{I_0}$, there are at most $1+\sum_{k=0}^{n-1} 2^k=2^n$ nonzero summands within the first parentheses on the right, and Cauchy--Schwarz and the orthogonality of martingale differences show that
\begin{equation*}
\begin{split}
  \Norm{E_{I_0}^\sigma f}{L^2(\sigma)} & +\sum_{I:I\subseteq I_0}\Norm{\Delta_{I}^\sigma f}{L^2(\sigma)}\\
  & \leq 2^{n/2}\Big(\Norm{E_{I_0}^\sigma f}{L^2(\sigma)}^2+\sum_{I:I\subseteq I_0}\Norm{\Delta_{I}^\sigma f}{L^2(\sigma)}^2\Big)^{1/2}
  = 2^{n/2}\Norm{f}{L^2(\sigma)}.
\end{split}
\end{equation*}
Combining this with a similar estimate for $g$ and $w$ gives us
\begin{equation*}
  \abs{B(f,g)}\leq 2\mathcal{W}\cdot 2^{n/2}\Norm{f}{L^2(\sigma)} \cdot 2^{n/2}\Norm{g}{L^2(w)}
  = 2^{n+1}\mathcal{W}\Norm{f}{L^2(\sigma)} \Norm{g}{L^2(w)},
\end{equation*}
which shows that $\mathcal{K}_n\leq 2^{n+1}\mathcal{W}<\infty$.
\end{proof}

To set the scene for the proof of the uniform boundedness of the $\mathcal{K}_n$, we pick two functions $f,g\in\mathscr{F}$ as in the definition of $\mathcal{K}_n$, so that $B(f,g)$ almost achieves the supremum defining $\mathcal{K}_n$:
\begin{equation}\label{eq:chooseMaximizers}
  \mathcal{K}_n\leq 2 B(f,g),\qquad\Norm{f}{L^2(\sigma)}=\Norm{g}{L^2(w)}=1.
\end{equation}

As in the proof of Lemma~\ref{lem:aPrioriBound}, let $I_0$ be the union of $2^n$ minimal intervals supporting both $f$ and $g$.
However, we now expand $f$ and $g$ in terms of a (randomly, eventually) shifted dyadic system $\mathscr{D}$. We only care about intervals of length at most $I_0$, so we may define $\mathscr{D}$ as $\mathscr{D}=\mathscr{D}+\omega$, where $\omega\in[0,\abs{I_0})$ is a multiple of the minimal length $2^{-m}$.

In the system $\mathscr{D}$, there are at most two intervals of length $\abs{I_0}$, say $I_1$ and $I_2$, that intersect $I_0$. Thus
\begin{equation*}
  f=\sum_{i=1}^2\Big(E_{I_i}^\sigma f+\sum_{\substack{I\in\mathscr{D} \\ I\subseteq I_i}}\Delta_I^\sigma f\Big)
  =:f_1^\omega+f_2^\omega,
\end{equation*}
and a similar expansion holds for $g$.

For both $f_i^\omega$, we make the further standard decomposition into ``good'' and ``bad'' parts,
\begin{equation*}
  f_i^\omega=E_{I_i}^\sigma f
   +\sum_{\substack{I\in\mathscr{D}; I\subseteq I_i \\ I \operatorname{good}}}\Delta_I^\sigma f
   +\sum_{\substack{I\in\mathscr{D}; I\subseteq I_i \\ I \operatorname{bad}}}\Delta_I^\sigma f  
   =:f_{i,0}^\omega+f_{i,\operatorname{good}}^\omega+f_{i,\operatorname{bad}}^\omega,
\end{equation*}
where an interval $I\in\mathscr{D}$ is called bad (with parameters $\gamma\in(0,1)$, $r\in\N$) if
\begin{equation*}
  \dist(I,J^c)\leq\abs{I}^\gamma\abs{J}^{1-\gamma}\quad
  \text{for some }J=I^{(k)},\quad k\geq r,
\end{equation*}
where $I^{(k)}$ refers to the $k$th dyadic ancestor of $I$ in the dyadic system $\mathscr{D}$, and we only consider $k$ such that $\abs{I^{(k)}}\leq\abs{I_0}$.

The important standard property is that
\begin{equation*}
  \Exp_\omega\Norm{f_{i,\operatorname{bad}}^\omega}{L^2(\sigma)}\leq \epsilon\Norm{f}{L^2(\sigma)}=\epsilon,
\end{equation*}
where $\epsilon=\epsilon(\gamma,r)\to 0$ as $r\to\infty$. We may fix $\gamma$ and $r$ so that, say, $\epsilon\leq\frac{1}{10}$.

Now we can make the following computations:
\begin{equation*}
  \mathcal{K}_n
  \leq 2\abs{B(f,g)}
  \leq 2\sum_{i=1}^2\abs{B(f_i^\omega,g_i^\omega)}+2\sum_{\substack{1\leq i,j\leq 2 \\ i\neq j }}\abs{B(f_i^\omega,g_j^\omega)}
\end{equation*}
and
\begin{equation*}
\begin{split}
  \abs{B(f_i^\omega,g_i^\omega)}
  &\leq\abs{B(f_{i,0}^\omega,g_i^\omega)}+\abs{B(f_{i,\operatorname{bad}}^\omega,g_i^\omega)} \\
  &\qquad +\abs{B(f_{i,\operatorname{good}}^\omega,g_{i,0}^\omega)}+\abs{B(f_{i,\operatorname{good}}^\omega,g_{i,\operatorname{bad}}^\omega)}
    +\abs{B(f_{i,\operatorname{good}}^\omega,g_{i,\good}^\omega)} \\
  &\leq\mathcal{H}\Norm{f_{i,0}^\omega}{L^2(\sigma)}\Norm{g_{i}^\omega}{L^2(w)}
    +\mathcal{K}_n\Norm{f_{i,\operatorname{bad}}^\omega}{L^2(\sigma)}\Norm{g_i^\omega}{L^2(w)}  \\
   &\qquad +\mathcal{H}^*\Norm{f_{i,\operatorname{good}}^\omega}{L^2(\sigma)}\Norm{g_{i,0}^\omega}{L^2(w)}
      +\mathcal{K}_n\Norm{f_{i,\operatorname{good}}^\omega}{L^2(\sigma)} \Norm{g_{i,\operatorname{bad}}^\omega}{L^2(w)} \\
  &\qquad        +\abs{B(f_{i,\operatorname{good}}^\omega,g_{i,\operatorname{good}}^\omega)} \\
  &\leq(\mathcal{H}+\mathcal{H}^*)
      +\mathcal{K}_n(\Norm{f_{i,\operatorname{bad}}^\omega}{L^2(\sigma)}+\Norm{g_{i,\operatorname{bad}}^\omega}{L^2(w)}) \\
   &\qquad +\abs{B(f_{i,\operatorname{good}}^\omega,g_{i,\operatorname{good}}^\omega)},
\end{split}
\end{equation*}
where we used the fact that all the splittings of our functions are orthogonal in the corresponding function spaces, and the functions $f,g$ have unit norm, as we assumed in \eqref{eq:chooseMaximizers}.

So, combining these bounds and taking expectations over the randomization parameter $\omega$, we find that
\begin{equation*}
\begin{split}
  \mathcal{K}_n
  \leq 4(\mathcal{H}+\mathcal{H}^*)+8\mathcal{K}_n\cdot\epsilon
  &+2\Exp_\omega\sum_{i=1}^2\abs{B(f_{i,\operatorname{good}}^\omega,g_{i,\operatorname{good}}^\omega)} \\
   &  +2\Exp_\omega\sum_{\substack{1\leq i,j\leq 2 \\ i\neq j }}\abs{B(f_i^\omega,g_j^\omega)},
\end{split}
\end{equation*}
or, using $8\mathcal{K}_n\cdot\epsilon\leq \frac45\mathcal{K}_n$ and absorbing this term on the left (here the finiteness of $\mathcal{K}_n$ is used),
\begin{equation}\label{eq:mainToProve}
\begin{split}
  \mathcal{K}_n\leq 20(\mathcal{H}+\mathcal{H}^*)
  &+10\Exp_\omega\sum_{i=1}^2\abs{B(f_{i,\operatorname{good}}^\omega,g_{i,\operatorname{good}}^\omega)} \\
   &+10\Exp_\omega\sum_{\substack{1\leq i,j\leq 2 \\ i\neq j }}\abs{B(f_i^\omega,g_j^\omega)}.
\end{split}
\end{equation}
Now the heart of the proof will be concerned with bounding the last two terms on the right. The randomization parameter $\omega$ is no longer relevant, but we will estimate these terms uniformly in $\omega$.

The term involving the good parts is the hard core of all boundedness results of singular integrals in our genre. The other part, involving the functions $f_i^\omega$ and $g_j^\omega$ with $i\neq j$, is not completely trivial either, although an order of magnitude easier than the hard core. The functions here are supported on disjoint adjacent intervals, so that the bilinear form is completely determined by its kernel, and the relevant estimate is in effect a positive kernel result, with no role for cancellation.

\section{Positive kernel theory}

This section is concerned with the part of the theory that only uses the upper bound $\frac{1}{\abs{x-y}}$ for the kernel of the Hilbert transform, with no role for cancellation. This will already be enough to estimate the last term in \eqref{eq:mainToProve}.

\subsection{Muckenhoupt's theory for the Hardy operator}

The two-weight problem for the Hardy operator was resolved by Muckenhoupt \cite{Muckenhoupt:72} as early as 1972. However, his result is not stated in the nowadays preferred dual-weight formulation, and is not literally applicable for our purposes in the generality of an arbitrary pair of measures. For this reason, we include the present section to reproduce Muckenhoupt's theory in a form suitable for the subsequent analysis.

We prove the result for general measures by adapting Muckenhoupt's original argument for measures with density. This strategy was suggested by Muckenhoupt as a possibility, although he preferred to treat his version of the general case by approximating with the continuous case instead. As pointed out by Muckenhoupt \cite[bottom of p.~36]{Muckenhoupt:72}, the direct argument requires the following inequality, which ``is true but the proof is not particularly simple.''

\begin{lemma}\label{lem:wxInf}
For any $t\in\R$, any measure $w$ on $[t,\infty)$ and $\alpha\in(0,1)$, we have
\begin{equation*}
  \int_{[t,\infty)}w[x,\infty)^{-\alpha}\ud w(x)\leq\frac{w[t,\infty)^{1-\alpha}}{1-\alpha}.
\end{equation*}
\end{lemma}

\begin{proof}
If $w[t,\infty)=\infty$, there is nothing to prove, so we assume that $w[x,\infty)\leq w[t,\infty)<\infty$ for all $x\geq t$. As $t\leq x_n\uparrow x$, we have $[x_n,\infty)\downarrow[x,\infty)$, thus $w[x_n,\infty)\downarrow w[x,\infty)$, and hence $w[x_n,\infty)^{-\alpha}\uparrow w[x,\infty)^{-\alpha}$.

For every $n\in\Z_+$, consider the grid of points $t^n_j:=t+j2^{-n}$, $j\in\N$, on $[t,\infty)$, and let $\phi_n(x)$ be the largest of these grid points such that $\phi_n(x)\leq x$. It follows that $\phi_n(x)\uparrow x$ and hence $w[\phi_n(x),\infty)^{-\alpha}\uparrow w[x,\infty)^{-\alpha}$ for every $x\in[t,\infty)$, and thus by monotone convergence
\begin{equation*}
  \int_{[t,\infty)}w[x,\infty)^{-\alpha}\ud w(x)
  =\lim_{n\to\infty}\int_{[t,\infty)}w[\phi_n(x),\infty)^{-\alpha}\ud w(x)
\end{equation*}
where
\begin{equation*}
\begin{split}
    \int_{[t,\infty)}w[\phi_n(x),\infty)^{-\alpha}\ud w(x) 
    &=\sum_{j=0}^\infty\int_{[t^n_j,t^n_{j+1})}w[t^n_j,\infty)^{-\alpha}\ud w(x) \\
    &=\sum_{j=0}^\infty  w[t^n_j,\infty)^{-\alpha}w[t^n_j,t^n_{j+1})
      =:\sum_{j=0}^\infty  w_j^{-\alpha}(w_j-w_{j+1}),
\end{split}
\end{equation*}
where (suppressing the fixed $n$), $w_j:=w[t^n_j,\infty)$.

By the mean value theorem, we have
\begin{equation*}
  w_j^{1-\alpha}-w_{j+1}^{1-\alpha}
  =(1-\alpha)\xi_j^{-\alpha}(w_j-w_{j+1})\geq(1-\alpha)w_j^{-\alpha}(w_j-w_{j+1}),
\end{equation*}
and hence
\begin{equation*}
  \sum_{j=0}^\infty  w_j^{-\alpha}(w_j-w_{j+1})
  \leq\frac{1}{1-\alpha}\sum_{j=0}^\infty(w_j^{1-\alpha}-w_{j+1}^{1-\alpha})
  \leq\frac{1}{1-\alpha}w_0^{1-\alpha}
  =\frac{w[t,\infty)^{1-\alpha}}{1-\alpha}
\end{equation*}
by summing a telescopic series. This completes the proof.
\end{proof}

\begin{lemma}\label{lem:w0x}
For any $t>0$, any measure $\sigma$ on $(0,t]$ and $\alpha\in(0,1)$, we have
\begin{equation*}
  \int_{(0,t]}\sigma(0,x]^{-\alpha}\ud\sigma(x)\leq\frac{\sigma(0,t]^{1-\alpha}}{1-\alpha}.
\end{equation*}
\end{lemma}

\begin{proof}
The proof is analogous to the previous one.
\end{proof}

Now we are ready for a dual-weight formulation of Muckenhoupt's inequality for the Hardy operator.

\begin{theorem}[Muckenhoupt]\label{thm:Muckenhoupt}
For a $\sigma$-finite measure $\sigma$ on $(0,\infty)$, and another measure $w$, the inequality
\begin{equation*}
  \BNorm{x\mapsto\int_{(0,x]}f\ud\sigma}{L^2((0,\infty),w)}
  \leq C\Norm{f}{L^2((0,\infty),\sigma)}
\end{equation*}
holds for all $f\in L^2(\sigma)$ if and only if
\begin{equation}\label{eq:HardyMuckConst}
  A:=\sup_{t>0}\sigma(0,t]^{1/2} w[t,\infty)^{1/2}<\infty.
\end{equation}
Moreover, the optimal constant $C$ satisfies
\begin{equation*}
  A\leq C\leq 2A.
\end{equation*}
\end{theorem}

As indicated, this Theorem is essentially due to Muckenhoupt~\cite{Muckenhoupt:72}. It is also stated more or less in the present form but without proof in \cite[Eqs.~(2.10--2.11)]{LSU:2wHilbertEnergy}.

\begin{proof}
``$\Rightarrow$''
Let $f=1_{E}$, with $E\subseteq(0,t]$. Then $\int_{(0,x]}f\ud\sigma=\sigma(E)$ for all $x\geq t$, and we obtain
\begin{equation*}
  \Norm{\sigma(E)\cdot 1_{[t,\infty)}}{L^2(w)}\leq C\Norm{1_{E}}{L^2(\sigma)},
\end{equation*}
i.e., $\sigma(E)w[t,\infty)^{1/2}\leq C\sigma(E)^{1/2}$. If $\sigma(E)<\infty$, we deduce $\sigma(E)^{1/2}w[t,\infty)^{1/2}\leq C$. By $\sigma$-finiteness, $(0,t]$ can be exhausted by countably many such sets, and we deduce that $\sigma(0,t]^{1/2}w[t,\infty)^{1/2}\leq C$.

``$\Leftarrow$'' Inside the integral
\begin{equation}\label{eq:HardyToEst}
   \int_{(0,x]}f(y)\ud\sigma(y)
\end{equation}
to be estimated, we wish to multiply and divide by $a(y):=\sigma(0,y]^{1/4}$. Let us check that this is legitimate.
 
 If $a(z)=\infty$, then $w[z,\infty)=0$ by assumption \eqref{eq:HardyMuckConst}. So our integration in $\ud w(x)$ only extends over $x\in(0,z)$. Thus the interval $(0,x]$ appearing in \eqref{eq:HardyToEst} is always smaller than $(0,z)$, so the point $z$ with $a(z)=\infty$ stays outside. If $a(z)=0$, then $\int_{(0,z]}f\ud\sigma=0$, so we can drop this part of the integral as well. Hence no zero-division problems occur, and we may proceed:
\begin{equation}\label{eq:HardyStart}
\begin{split}
  \int_{(0,\infty)} &\Big(\int_{(0,x]}f(y)a(y)a(y)^{-1}\ud\sigma(y)\Big)^2 \ud w(x) \\
  &\leq\int_{(0,\infty)}\int_{(0,x]}f(y)^2 a(y)^2\ud\sigma(y)\int_{(0,x]}a(z)^{-2}\ud\sigma(z) \ud w(x) \\
  &=\int_{(0,\infty)}f(y)^2\sigma(0,y]^{1/2}
     \Big( \int_{[ y,\infty)} \int_{(0,x]}\sigma(0,z]^{-1/2}\ud\sigma(z) \ud w(x)  \Big) \ud\sigma(y). \\
\end{split}
\end{equation}
The quantity in parentheses is estimated by
\begin{equation*}
\begin{split}
     \int_{[y,\infty)} & \int_{(0,x]}\sigma(0,z]^{-1/2}\ud\sigma(z) \ud w(x)
     \leq\int_{[ y,\infty)} 2\sigma(0,x]^{1/2} \ud w(x) \\
    & \leq 2A\int_{[ y,\infty)}w[x,\infty)^{-1/2}\ud w(x)
    \leq 2A\cdot 2w[ y,\infty)^{1/2},
\end{split}
\end{equation*}
where the three steps used, in this order, Lemma \ref{lem:w0x}, assumption \eqref{eq:HardyMuckConst}, and Lemma \ref{lem:wxInf}.
Substituting back to \eqref{eq:HardyStart}, we get
\begin{equation*}
\begin{split}
  \int_{(0,\infty)} &\Big(\int_{(0,x]}f(y)\ud\sigma(y)\Big)^2 \ud w(x)  \\
  &\leq\int_{(0,\infty)}f(y)^2\sigma(0,y]^{1/2} 4A w[ y,\infty)^{1/2}\ud \sigma(y)
  \leq 4A^2\int_{(0,\infty)}f(y)^2\ud\sigma(y),
\end{split}  
\end{equation*}
where \eqref{eq:HardyMuckConst} was used another time.
\end{proof}


\subsection{The Hilbert transform on complementary half-lines}

The following simple Corollary to Theorem~\ref{thm:Muckenhoupt} makes the bridge from the Hardy operator to the Hilbert transform.

\begin{corollary}\label{cor:Muckenhoupt}
Let $\sigma,w$ be $\sigma$-finite measures on $(0,\infty)$.
The estimate
\begin{equation*}
  \Babs{\iint_{(0,\infty)^2}\frac{1}{x+y}f(y)\ud\sigma(y)g(x)\ud w(x)}
  \leq C\Norm{f}{L^2(\sigma)}\Norm{g}{L^2(w)}
\end{equation*}
holds, if and only if
\begin{equation*}
  A:=\sup_{t>0}\Big(\sigma(0,t]\int_{[t,\infty)}\frac{\ud w(x)}{x^2}\Big)^{1/2}+
  \sup_{t>0}\Big(w(0,t]\int_{[t,\infty)}\frac{\ud \sigma(x)}{x^2}\Big)^{1/2}<\infty.
\end{equation*}
Moreover, the optimal constant $C$ satisfies $\frac14A\leq C\leq 2A$.
\end{corollary}

\begin{proof}
It suffices to consider nonnegative $f,g$. Writing
\begin{equation*}
\begin{split}
  \Lambda(f,g) &  :=\iint_{(0,\infty)^2}\frac{1}{x+y}f(y)\ud\sigma(y)g(x)\ud w(x), \\
  \Lambda_1(f,g) &  :=\int_{(0,\infty)}\frac{1}{x}\int_{(0,x]}f(y)\ud\sigma(y)g(x)\ud w(x), \\
  \Lambda_2(f,g) &  :=\int_{(0,\infty)}\frac{1}{y}\int_{(0,y]}g(x)\ud w(x)f(y)\ud \sigma(y),
\end{split}
\end{equation*}
it is immediate that
\begin{equation*}
  \frac12\max_{i=1,2}\Lambda_i(f,g)\leq\Lambda(f,g)\leq\Lambda_1(f,g)+\Lambda_2(f,g).
\end{equation*}

Let $C_i$ be the best constant in the inequality $\abs{\Lambda_i(f,g)}\leq C_i\Norm{f}{L^2(\sigma)}\Norm{g}{L^2(w)}$. It follows that $\frac12\max\{C_1,C_2\}\leq C\leq C_1+C_2$. By Cauchy--Schwarz, $C_1$ is also the best constant in the estimate
\begin{equation*}
   \Big(\int_{(0,\infty)}\Big(\int_{(0,x]}f(y)\ud\sigma(y)\Big)^2\frac{\ud w(x)}{x^2}\Big)^{1/2}\leq C_1\Norm{f}{L^2(\sigma)},
\end{equation*}
which by Muckenhoupt's theorem (applied to $\ud w(x)/x^2$ in place of $\ud w(x)$) satisfies
\begin{equation*}
  A_1\leq C_1\leq 2A_1,\qquad A_1:=\Big(\sup_{t>0}\sigma(0,t]\int_{[t,\infty)}\frac{\ud w(x)}{x^2}\Big)^{1/2}.
\end{equation*}
An analogous claim is true for $C_2$ and $A_2$ defined with the roles of $\sigma$ and $w$ reversed. Observing that $A=A_1+A_2$, we find that
\begin{equation*}
  \frac14 A 
  \leq\frac12\max\{A_1,A_2\}
  \leq\frac12\max\{C_1,C_2\}
  \leq C\leq C_1+C_2\leq 2A.\qedhere 
\end{equation*}
\end{proof}

We are now ready to give a full characterization of the restricted two-weight boundedness of the Hilbert transform, where its bilinear form is applied to functions supported on complementary half-lines. It is not surprising that this restricted boundedness is determined by the $A_2$ condition alone. Still, this result appears to be unrecorded in the earlier literature on the two-weight problem.

\begin{theorem}\label{thm:complement}
Let $\sigma$ and $w$ be two Radon measures.
The estimate
\begin{equation*}
  \abs{B(f,g)}\leq C\Norm{f}{L^2(\sigma)}\Norm{g}{L^2(w)}
\end{equation*}
holds uniformly for all $f,g\in\mathscr{F}$ supported on complementary half-lines,
if and only if the measures satisfy the $A_2$ conditions \eqref{eq:newA2}.
Moreover, the best constant satisfies
\begin{equation*}
  C\eqsim[\sigma,w]_{A_2}^*+[w,\sigma]_{A_2}^*
  \eqsim\mathcal{H}_{\operatorname{off}}+\mathcal{H}_{\operatorname{off}}^*,
\end{equation*}
where the last two constants are defined as in \eqref{eq:HtestOff}.
\end{theorem}

Note that the Theorem also applies (and will often be applied) to the case that $\spt f\subseteq I$ and $\spt g\subseteq I^c$ for some finite interval $I$. Indeed, we just split $g$ into the two parts supported to the left or to the right of $I$.

\begin{proof}
It is immediate that $\mathcal{H}_{\operatorname{off}}+\mathcal{H}_{\operatorname{off}}^*\lesssim C$, and we proved in Lemma~\ref{lem:A2nec} that $[\sigma,w]_{A_2}^*+[w,\sigma]_{A_2}^*\lesssim\mathcal{H}_{\operatorname{off}}+\mathcal{H}_{\operatorname{off}}^*$. Hence it only remain to show that $C\lesssim [\sigma,w]_{A_2}^*+[w,\sigma]_{A_2}^*$.

Clearly the bound on $B(f,g)$ holds for all $f,g$ if and only if it holds for nonnegative $f,g$.
Suppose for definiteness that $\spt f\subset(-\infty,a)$ and $\spt g\subset(a,\infty)$. (The case that, for example, $\spt f\subset(-\infty,a)$ and $\spt g\subset[a,\infty)$ is easily reduced to $\spt f\subset(-\infty,a')$ and $\spt g\subset(a',\infty)$ for a slightly smaller $a'<a$.) Then
\begin{equation}\label{eq:disjHTbegin}
\begin{split}
  B(f,g) &=\int_{(0,\infty)}\int_{(0,\infty)}\frac{1}{x+y} f(a-y)g(a+x)\ud w(a+x)\ud\sigma(a-y) \\
   &=:\int_{(0,\infty)}\int_{(0,\infty)}\frac{1}{x+y} f_a(y)g^a(x)\ud w^a(x)\ud\sigma_a(y).
\end{split}
\end{equation}
The constant $c$ in the upper bound $c\Norm{f_a}{L^2(\sigma_a)}\Norm{g^a}{L^2(w^a)}=c\Norm{f}{L^2(\sigma)}\Norm{g}{L^2(w)}$ of this expression is characterized by Corollary~\ref{cor:Muckenhoupt} as
\begin{equation*}
\begin{split}
  c^2 &\eqsim\sup_{t>0}\sigma_a(0,t]\int_{[t,\infty)}\frac{\ud w^a(x)}{x^2}+\sup_{t>0}w^a(0,t]\int_{[t,\infty)}\frac{\ud\sigma_a(y)}{y^2} \\
     &=\sup_{t>0}\sigma[a-t,a)\int_{[a+t,\infty)}\frac{\ud w(x)}{(x-a)^2}+\sup_{t>0}w(a,a+t]\int_{(-\infty,a-t]}\frac{\ud\sigma(y)}{(y-a)^2}.
\end{split}
\end{equation*}
Observing that we need to consider this estimate for arbitrary $a\in\R$, and arbitrary order of the supports of $f$ and $g$, we deduce that
\begin{equation*}
  C^2 \lesssim \sup_I \sigma(I)\int_{(3I)^c}\frac{\ud w(x)}{\dist(x,I)^2}+\sup_I w(I)\int_{(3I)^c}\frac{\ud w(x)}{\dist(x,I)^2}.
\end{equation*}
It is immediate that this is dominated by a constant multiple of $([\sigma,w]_{A_2}^*+[w,\sigma]_{A_2}^*)^2$.
\end{proof}

In particular, Theorem~\ref{thm:complement} provides the required bound for the last part in \eqref{eq:mainToProve}:
\begin{equation*}
  \abs{B(f_i^\omega,g_j^\omega)}
  \lesssim([\sigma,w]_{A_2}^*+[w,\sigma]_{A_2}^*)\Norm{f_i^\omega}{L^2(\sigma)}\Norm{g_j^\omega}{L^2(\sigma)},
\end{equation*}
since the functions here are supported on disjoint intervals $I_i$ and $I_j$.

We can also already complete the easier equivalence of Theorem~\ref{thm:main}:

\begin{corollary}\label{cor:complement}
Let $\sigma$ and $w$ be two Radon measures on $\R$. Then the global testing conditions \eqref{eq:HtestGlobal} hold, if and only if the local testing conditions \eqref{eq:HtestLocal} and the Muckenhoupt--Poisson $A_2$ conditions \eqref{eq:newA2} hold. Moreover, we have
\begin{equation*}
   \mathcal{H}_{\operatorname{glob}}+\mathcal{H}_{\operatorname{glob}}^*
   \eqsim \mathcal{H}+\mathcal{H}^*+[\sigma,w]_{A_2}^*+[w,\sigma]_{A_2}^*.
\end{equation*}
\end{corollary}

\begin{proof}
It is immediate from the definitions that
\begin{equation*}
   \mathcal{H}_{\operatorname{glob}}\eqsim\mathcal{H}+\mathcal{H}_{\operatorname{off}},\qquad
   \mathcal{H}_{\operatorname{glob}}^*\eqsim\mathcal{H}^*+\mathcal{H}_{\operatorname{off}}^*.
\end{equation*}
The Corollary then follows from Theorem~\ref{thm:complement}.
\end{proof}

\section{The standard non-homogeneous $T(1)$ argument}

We now start the analysis of the main term arising from the initial reduction, the middle term on the right of \eqref{eq:mainToProve}. We drop the indices $\omega$, $i$, and ``good'', simply considering two functions $f$ and $g$ as follows:

\begin{assumption}\label{ass:fg}
The functions $f\in L^2(\sigma)$ and $g\in L^2(w)$ are supported on a common interval $I_0\in\mathscr{D}$, constant on its subintervals of length $2^{-n}\abs{I_0}$, of average zero over $I_0$, and such that $\Delta_I^\sigma f$ and $\Delta_I^w g$ are non-zero only for good intervals $I\in\mathscr{D}$. 
\end{assumption}

In particular, these functions have finite expansions
\begin{equation*}
  f=\sum_I \Delta_I^\sigma f,\qquad g=\sum_I \Delta_I^w g,
\end{equation*}
where the summations may be restricted to $I\subseteq I_0$, $\abs{I}>2^{-n}\abs{I_0}$, and good intervals only. The finiteness ensures that reorganizing these summations is never an issue.

We need to estimate the bilinear form
\begin{equation*}
  B(f,g)
  =\sum_{I,J} B(\Delta_I^\sigma f,\Delta_J^w g),
\end{equation*}
and we concentrate on the ``lower'' half
\begin{equation*}
  \sum_{I,J:\abs{I}\leq\abs{J}}=\sum_{I,J:I\subseteq J}+\sum_{I,J:I\subseteq 3J\setminus J}
  +\sum_{\substack{I,J:\abs{I}\leq\abs{J}\\ I\subseteq (3J)^c}},
\end{equation*}
as the ``upper'' half with $\abs{I}>\abs{J}$ is handled analogously. (In choosing to concentrate on $\abs{I}\leq\abs{J}$, we follow the convention of several papers in non-homogeneous analysis, notably \cite{NTV:Tb}, but deviate from Lacey et al.~\cite{LSSU:2wHilbert}, who explicitly consider the case $\abs{I}\geq\abs{J}$.)

The analysis in this section is ``standard'', but cannot be literally quoted from an earlier source, as we need to ensure that we only apply the $A_2$ condition in the form \eqref{eq:newA2}.

\subsection{Separated part}
We first treat the separated part
\begin{equation*}
\begin{split}
   \sum_{\substack{I,J:\abs{I}\leq\abs{J}\\ I\subseteq (3J)^c}}
   &=\sum_{i=0}^\infty\sum_{k=1}^\infty \sum_{2^k\leq\abs{m}< 2^{k+1}}\sum_J\sum_{I:I^{(i)}=J\dot+m} \\
   &=\sum_{i=0}^\infty\sum_{k=1}^\infty\sum_K \sum_{J:J^{(k)}=K}\sum_{2^k\leq\abs{m}< 2^{k+1}}\sum_{I:I^{(i)}=J\dot+m}.
\end{split}
\end{equation*}
Here it is worth observing that $J\dot+ m\subseteq 5K\setminus K$.

\begin{lemma}\label{lem:standardKernel}
For disjointly supported $f\in L^1(\sigma)$ with $\int f\ud\sigma=0$, and $g\in L^1(w)$,  we have
\begin{equation*}
  \abs{B(f,g)}\lesssim\frac{\operatorname{diam}(\spt f)}{\dist(\spt f,\spt g)^2}\Norm{f}{L^1(\sigma)}\Norm{g}{L^1(w)}.
\end{equation*}
\end{lemma}

\begin{proof}
This is completely standard, using the bound for the derivative of the Hilbert kernel.
\end{proof}

Then, by the additivity of the $L^1$ norm on disjoint functions, we deduce that
\begin{equation*}
\begin{split}
   \sum_{J:J^{(k)}=K} &\sum_{2^k\leq\abs{m}< 2^{k+1}}\sum_{I:I^{(i)}=J\dot+m}\abs{B(\Delta_I^\sigma f,\Delta_J^w g)} \\
  &\lesssim 2^{-i-k}\frac{1}{\abs{K}}
      \sum_{J:J^{(k)}=K}\Norm{\Delta_J^w g}{L^1(w)}\sum_{m:1\leq\abs{m}\leq 2}\sum_{I:I^{(i+k)}=K\dot+m}\Norm{\Delta_I^\sigma f}{L^1(\sigma)} \\
  &=2^{-i-k}\frac{1}{\abs{K}}\Norm{\Delta_K^{w,k} g}{L^1(w)}\sum_{m:1\leq\abs{m}\leq 2}\Norm{\Delta_{K\dot+m}^{\sigma,i+k}f}{L^1(\sigma)},
\end{split}
\end{equation*}
where
\begin{equation*}
  \Delta_K^{w,k}g:=\sum_{J:J^{(k)}=K}\Delta_J^w g
\end{equation*}
and $\Delta_{K\dot+m}^{\sigma,i+k}f$ is defined analogously.

Next, we use Cauchy--Schwarz to estimate
\begin{equation*}
  \Norm{\Delta_K^{w,k} g}{L^1(w)}\Norm{\Delta_{K\dot+m}^{\sigma,i+k}f}{L^1(\sigma)}
  \leq w(K)^{1/2}\sigma(K\dot+m)^{1/2}\Norm{\Delta_K^{w,k} g}{L^2(w)}\Norm{\Delta_{K\dot+m}^{\sigma,i+k}f}{L^2(\sigma)},
\end{equation*}
and then
\begin{equation*}
  w(K)^{1/2}\sigma(K\dot+m)^{1/2}\lesssim[\sigma,w]_{A_2}\abs{K},\qquad\abs{m}=1,2.
\end{equation*}
So altogether we have checked that
\begin{equation*}
\begin{split}
  \sum_{\substack{I,J:\abs{I}\leq\abs{J}\\ I\subseteq(3J)^c}} &\abs{B(\Delta_I^\sigma f,\Delta_J^w g)} \\
  &\lesssim[\sigma,w]_{A_2}\sum_{i=0}^\infty\sum_{k=1}^\infty 2^{-i-k}\sum_{m:1\leq\abs{m}\leq 2}\sum_K
  \Norm{\Delta_K^{w,k}g}{L^2(w)}\Norm{\Delta_{K\dot+m}^{\sigma,i+k}f}{L^2(\sigma)}.
\end{split}
\end{equation*}
The inner sum is bounded by
\begin{equation*}
  \Big(\sum_K \Norm{\Delta_K^{w,k}g}{L^2(w)}^2\Big)^{1/2}
  \Big(\sum_K \Norm{\Delta_{K\dot+m}^{\sigma,i+k}f}{L^2(\sigma)}^2\Big)^{1/2}
  \leq\Norm{g}{L^2(w)}\Norm{f}{L^2(\sigma)}
\end{equation*}
by orthogonality, and it remains to sum up the geometric series in $i$ and $k$ and the finite sum over $\abs{m}=1,2$ to conclude that
\begin{equation*}
   \sum_{\substack{I,J:\abs{I}\leq\abs{J}\\ I\subseteq(3J)^c}} \abs{B(\Delta_I^\sigma f,\Delta_J^w g)}
   \lesssim[\sigma,w]_{A_2}\Norm{f}{L^2(\sigma)}\Norm{g}{L^2(w)}.
\end{equation*}

\subsection{Adjacent part}
We turn our attention to
\begin{equation}\label{eq:adjacentPart}
\begin{split}
  \sum_{I,J:I\subseteq 3J\setminus J}B(\Delta_I^\sigma f,\Delta_J^w g)
  &=\sum_{m=\pm 1}\sum_{i=0}^\infty\sum_{I:I^{(i)}=J\dot+m}B(\Delta_I^\sigma f,\Delta_J^w g) \\
  &=\sum_{m=\pm 1}\sum_{i=0}^\infty B(\Delta_{J\dot+m}^{\sigma,i}f,\Delta_J^w g).
\end{split}
\end{equation}

Before proceeding, we record a useful lemma:

\begin{lemma}\label{lem:pairWithDelta}
For all $f,g\in\mathscr{F}$ and $J\in\mathscr{D}$, we have
\begin{equation*}
   \abs{B(f,\Delta_J^w g)}
   \lesssim ([\sigma,w]_{A_2}^*+[w,\sigma]_{A_2}^*+C)\Norm{f}{L^2(\sigma)}\Norm{\Delta_J^w g}{L^2(w)},
\end{equation*}
where
\begin{equation*}
  C=\begin{cases} 0, & \text{if }\spt f\subseteq J^c, \\ \mathcal{H}^*, & \text{else}.\end{cases}
\end{equation*}
\end{lemma}

\begin{proof}
Writing $\Delta_J^w g=\sum_{u\in\{\operatorname{left},\operatorname{right}\}}\ave{\Delta_J^w g}_{J_u}^w 1_{J_u}$, it suffices to prove the estimate with $1_{J_u}$ in place of $\Delta_J^w g$. We also split $f=1_{J_u^c}f+1_{J_u} f$. Then
\begin{equation*}
  \abs{B(1_{J_u^c}f,1_{J_u})}
  \leq\mathcal{H}_{\operatorname{off}}^*\Norm{f}{L^2(\sigma)}w(J_u)^{1/2}
  \lesssim([\sigma,w]_{A_2}^*+[w,\sigma]_{A_2}^*)\Norm{f}{L^2(\sigma)}w(J_u)^{1/2}
\end{equation*}
by definition \eqref{eq:HtestOff} and Theorem~\ref{thm:complement}, and
\begin{equation*}
  \abs{B(1_{J_u}f,1_{J_u})}\leq\mathcal{H}^*\Norm{f}{L^2(\sigma)}w(J_u)^{1/2}
\end{equation*}
by definition. Of course this second term does not appear if $\spt f\subseteq J^c$.
\end{proof}

We return to the analysis of \eqref{eq:adjacentPart}.
For any fixed $i$, using the disjointness of the sets $J$ and $J\dot+m$, Lemma~\ref{lem:pairWithDelta} gives the following immediate estimate:
\begin{equation*}
    \abs{B(\Delta_{J\dot+m}^{\sigma,i}f,\Delta_J^w g)}
    \lesssim ([\sigma,w]_{A_2}^{*}+[w,\sigma]_{A_2}^*)\Norm{\Delta_{J\dot+m}^{\sigma,i}f}{L^2(\sigma)}\Norm{\Delta_J^w g}{L^2(w)},
\end{equation*}
and this can be summed over $J$ by Cauchy--Schwarz and orthogonality to deduce that
\begin{equation*}
   \sum_J\abs{B(\Delta_{J\dot+m}^{\sigma,i}f,\Delta_J^w g)}
   \lesssim([\sigma,w]_{A_2}^{*}+[w,\sigma]_{A_2}^*)\Norm{f}{L^2(\sigma)}\Norm{g}{L^2(w)}.
\end{equation*}
This is not enough to sum up the infinite series in $i$, but we can use this for the $r+1$ first terms, so that we are reduced to estimating
\begin{equation*}
  \sum_{m=\pm 1}\sum_{i>r}^\infty\sum_{I:I^{(i)}=J\dot+m}B(\Delta_I^\sigma f,\Delta_J^w g).
\end{equation*}
Here the goodness of $f$ becomes available to us, so we know that $\dist(I,J)\geq\abs{I}^\gamma\abs{J}^{1-\gamma}$ for the relevant nonzero terms. Combined with Lemma~\ref{lem:standardKernel}, this gives
\begin{equation*}
  \abs{B(\Delta_I^\sigma f,\Delta_J^w g)}
  \lesssim\frac{\abs{I}}{\abs{I}^{2\gamma}\abs{J}^{2(1-\gamma)}}\Norm{\Delta_I^\sigma f}{L^1(\sigma)}\Norm{\Delta_J^w g}{L^1(w)},
\end{equation*}
and then
\begin{equation*}
\begin{split}
  \sum_{I:I^{(i)}=J\dot+m} &\abs{B(\Delta_I^\sigma f,\Delta_J^w g)}
  \lesssim 2^{-i(1-2\gamma)}\frac{1}{\abs{J}}\Norm{\Delta_{J\dot+m}^{\sigma,i} f}{L^1(\sigma)}\Norm{\Delta_J^w g}{L^1(w)} \\
  &\lesssim 2^{-i(1-2\gamma)}\frac{1}{\abs{J}}\sigma(J\dot+m)^{1/2}w(J)^{1/2}
    \Norm{\Delta_{J\dot+m}^{\sigma,i} f}{L^2(\sigma)}\Norm{\Delta_J^w g}{L^2(w)} \\
    &\lesssim 2^{-i(1-2\gamma)}[\sigma,w]_{A_2}
    \Norm{\Delta_{J\dot+m}^{\sigma,i} f}{L^2(\sigma)}\Norm{\Delta_J^w g}{L^2(w)}.
\end{split}
\end{equation*}
This can be summed over $J$ as before, and over $i$ by geometric convergence.

\subsection{Nested part}
We are left with
\begin{equation*}
  \sum_{I,J:I\subseteq J}B(\Delta_I^\sigma j,\Delta_J^w g)
  =\sum_{i=0}^\infty\sum_J\sum_{I:I^{(i)}=J}B(\Delta_I^\sigma f,\Delta_J^w g)
  =\sum_{i=0}^\infty\sum_J B(\Delta_J^{\sigma,i}f,\Delta_J^w g).
\end{equation*}
Again, we will check that any fixed $i$ is under control. From Lemma~\ref{lem:pairWithDelta}, we have
\begin{equation*}
  \abs{B(\Delta_J^{\sigma,i}f,\Delta_J^w g)}
  \lesssim(\mathcal{H}^*+[\sigma,w]_{A_2}^{*}+[w,\sigma]_{A_2}^*)\Norm{\Delta_J^{\sigma,i}f}{L^2(\sigma)}\Norm{\Delta_J^w g}{L^2(w)}.
\end{equation*}
By the usual Cauchy--Schwarz and orthogonality in the sum over $J$, we find that
\begin{equation*}
  \sum_J\abs{B(\Delta_J^{\sigma,i}f,\Delta_J^w g)}
  \lesssim(\mathcal{H}^*+[\sigma,w]_{A_2}^{*}+[w,\sigma]_{A_2}^*)\Norm{f}{L^2(\sigma)}\Norm{g}{L^2(w)}.
\end{equation*}
We apply this to $i=0,1,\ldots,r$, so that we are left with
\begin{equation*}
  \sum_{i>r}\sum_J\sum_{I:I^{(i)}=J}B(\Delta_I^\sigma f,\Delta_J^w g)
  =\sum_{j\geq r}\sum_K\sum_{I:I^{(j)}=K}B(\Delta_I^\sigma f,\Delta_{K^{(1)}}^w g),
\end{equation*}
where we chose the child $K$ of $J$ that contains $I$ as the new summation variable.

There is once again the disjoint part, where Lemma~\ref{lem:standardKernel} applies:
\begin{equation*}
   \abs{B(\Delta_I^\sigma f, 1_{K^c}\Delta_{K^{(1)}}^w g)}
   \lesssim\frac{\abs{I}}{\dist(I,K^c)^2}\Norm{\Delta_I^\sigma f}{L^1(\sigma)}\Norm{1_{K^c}\Delta_{K^{(1)}}^w g}{L^1(w)}.
\end{equation*}
Thus
\begin{equation*}
\begin{split}
  \sum_{I:I^{(j)}=K} &\abs{B(\Delta_I^\sigma f, 1_{K^c}\Delta_{K^{(1)}}^w g)}
  \lesssim 2^{-j(1-2\gamma)}\frac{1}{\abs{K}}\Norm{\Delta_K^{\sigma,j}f}{L^1(\sigma)}\Norm{1_{K^c}\Delta_{K^{(1)}}^w g}{L^1(w)} \\
  &\leq 2^{-j(1-2\gamma)}\frac{1}{\abs{K}}\sigma(K)^{1/2}w(K^{(1)}\setminus K)^{1/2}
     \Norm{\Delta_K^{\sigma,j}f}{L^2(\sigma)}\Norm{1_{K^c}\Delta_{K^{(1)}}^w g}{L^2(w)} \\
  &\leq 2^{-j(1-2\gamma)}[\sigma,w]_{A_2}
     \Norm{\Delta_K^{\sigma,j}f}{L^2(\sigma)}\Norm{1_{K^c}\Delta_{K^{(1)}}^w g}{L^2(w)},
\end{split}
\end{equation*}
and this is summed over $K$ and $j$ as before.

We are only left with
\begin{equation*}
  \sum_{j\geq r}\sum_K\sum_{I:I^{(j)}=K}B(\Delta_I^\sigma f, 1_K)\ave{\Delta_{K^{(1)}}^w g}_K^w
  =\sum_J\sum_{I:I\Subset J}B(\Delta_I^\sigma f, 1_{J_I})\ave{\Delta_{J}^w g}_{J_I}^w,
\end{equation*}
where we shifted back to the summation variable $J=K^{(1)}$, and $J_I=K$ denotes the half of $J$ that contains $I$; the notation $I\Subset J$ is abbreviation for $I\subset J$ and $\abs{I}<2^{-r}\abs{J}$.

Denoting the right hand side in the previous display by
\begin{equation}\label{eq:Bbelow}
   B_{\operatorname{below}}(f,g) := \sum_J\sum_{I:I\Subset J}B(\Delta_I^\sigma f, 1_{J_I})\ave{\Delta_{J}^w g}_{J_I}^w,
\end{equation}
and analogously
\begin{equation}\label{eq:Babove}
   B_{\operatorname{above}}(f,g) := \sum_I\sum_{J:J\Subset I}\ave{\Delta_{I}^\sigma f}_{I_J}^\sigma B(1_{I_J},\Delta_J^w g),
\end{equation}
the results of this section, and the symmetric considerations for $\abs{I}>\abs{J}$, may be summarized as follows:

\begin{proposition}
For functions $f$ and $g$ as in Assumption~\ref{ass:fg}, we have
\begin{equation*}
\begin{split}
  &\abs{B(f,g)-B_{\operatorname{above}}(f,g)-B_{\operatorname{below}}(f,g)} \\
  &\qquad\lesssim([\sigma,w]_{A_2}^*+[w,\sigma]_{A_2}^*+\mathcal{H}+\mathcal{H}^*)\Norm{f}{L^2(\sigma)}\Norm{g}{L^2(w)}.
\end{split}
\end{equation*}
\end{proposition}

We are thus left with estimating the new bilinear forms \eqref{eq:Bbelow} and \eqref{eq:Babove}, and by symmetry we concentrate on \eqref{eq:Bbelow} only.

\section{Stopping construction; reduction to local and tail forms}

We continue where we left above, the analysis of the form \eqref{eq:Bbelow}. Our next task is to rearrange this sum with the help of suitable \emph{stopping intervals}. However, we only specify the actual stopping condition later on, and proceed for the moment with a generic assumption:

\begin{assumption}[Stopping intervals]
We assume that $\mathscr{S}\subseteq\mathscr{D}$ is a subcollection containing a maximal interval $S_0$, where $S_0:=I_0$ is the initial interval that supports both $f$ and $g$. We also assume that $\mathscr{S}$ satisfies the Carleson condition
\begin{equation}\label{eq:Carleson}
  \sum_{\substack{S\in\mathscr{S}\\ S\subseteq J}}w(S)\lesssim w(J),
\end{equation}
where the implied constant is absolute.
\end{assumption}

Note that we introduce the alternative notation $S_0:=I_0$ simply to emphasize the role of the initial interval $I_0$ as a stopping interval.

\begin{remark}
By the well-known dyadic Carleson embedding theorem, condition~\eqref{eq:Carleson} implies that
\begin{equation}\label{eq:StoSatisfy1}
  \sum_{S\in\mathscr{S}}w(S)\abs{\ave{g}_S^w}^2\lesssim\Norm{g}{L^2(w)}^2
\end{equation}
for all $g\in L^2(w)$.
\end{remark}

 For every dyadic interval $I\subseteq I_0$, we define the \emph{stopping parent}
\begin{equation*}
  \pi I:=\pi_{\mathscr{S}}I:=\min\{S\in\mathscr{S}:S\supseteq I\}.
\end{equation*}
Note that the minimum of a discrete sequence of nested sets above could also be written as the intersection. We shall use a similar notation for some other systems of cubes, but reserve the shorthand $\pi=\pi_{\mathscr{S}}$ for our ``standard'' stopping cubes.

We also consider the projections
\begin{equation*}
  P^\sigma_S f:=\sum_{I:\pi I=S}\Delta_I^\sigma f,\qquad
  \tilde P^\sigma_{S}f:=\sum_{I:\pi I^{(r)}=S}\Delta_I^\sigma f,\qquad
  P^w_S g:=\sum_{J:\pi J=S}\Delta_J^w g,
\end{equation*}
so that
\begin{equation*}
  f=\sum_{S\in\mathscr{S}}\tilde P_{S}^\sigma f+\sum_{\substack{I:I\subseteq S_0\\ I\not\Subset S_0}}\Delta_I^\sigma f,\qquad
  g=\sum_{S\in\mathscr{S}}P_S^w g.
\end{equation*}
For the ``high'' intervals $I\not\Subset S_0$, there is no $J\subseteq S_0$ with $I\Subset J$, and hence
\begin{equation}\label{eq:splitBbelow}
\begin{split}
  B_{\operatorname{below}}(f,g)
  &=\sum_{S,S'\in\mathscr{S}}B_{\operatorname{below}}(\tilde P^\sigma_{S}f,P^w_{S'} g) \\
  &=\sum_S B_{\operatorname{below}}(\tilde P^\sigma_{S} f,P^w_S g)
    +\sum_{S,S':S'\supsetneq S}B_{\operatorname{below}}(\tilde P^\sigma_{S} f,P^w_{S'} g),
\end{split}
\end{equation}
where we also observed  that $B_{\operatorname{below}}(\tilde P^\sigma_{S} f,P^w_{S'} g)=0$ if $S\cap S'=\varnothing$ or $S\supsetneq S'$.
The first term on the right of \eqref{eq:splitBbelow} is the \emph{local form}
\begin{equation}\label{eq:Blocal}
  B_{\operatorname{local}}(f,g):=\sum_S B_{\operatorname{below}}(\tilde P^\sigma_{S} f,P^w_S g),
\end{equation}
which we postpone for later treatment.

We now consider the second part on the right of \eqref{eq:splitBbelow}. This can be written as
\begin{equation}\label{eq:BbelowUnequal}
\begin{split}
  \sum_{S,S':S'\supsetneq S}B_{\operatorname{below}}(\tilde P^\sigma_{S} f,P^w_{S'} g)
  =\sum_{S\in\mathscr{S}}\sum_{J:J\supsetneq S}B(\tilde P_S^\sigma f, 1_{J_S})\ave{\Delta_J^w g}_{J_S}^w.
\end{split}
\end{equation}
We split $1_{J_S}=1_{J_S\setminus S}+1_S$ and control the latter part first. With $1_S$ in place of $1_{J_S}$ above, only the last factor depends on $J$, and we have
\begin{equation*}
  \sum_{J:J\supsetneq S}\ave{\Delta_J^w g}_{J_S}^w=\ave{g}_S^w.
\end{equation*}
We thus need to estimate
\begin{equation*}
\begin{split}
  \Babs{\sum_{S\in\mathscr{S}}B(\tilde P_S^\sigma f, 1_S)\ave{g}_S^w}
  &\leq\sum_{S\in\mathscr{S}}\mathcal{H}^*\Norm{\tilde P_S^{\sigma}f}{L^2(\sigma)}w(S)^{1/2}\abs{\ave{g}_S^w} \\
  &\leq\mathcal{H}^*\Big(\sum_{S\in\mathscr{S}}\Norm{\tilde P_S^\sigma f}{L^2(\sigma)}^2\Big)^{1/2}
     \Big(\sum_{S\in\mathscr{S}}w(S)\abs{\ave{g}_S^w}^2\Big)^{1/2}.
\end{split}
\end{equation*}
The factor involving $f$ is bounded by $\Norm{f}{L^2(\sigma)}$ by orthogonality, and the factor involving $g$ by $\Norm{g}{L^2(w)}$ by the Carleson embedding theorem as recorded in \eqref{eq:StoSatisfy1}. We thus deduce that
\begin{equation*}
   \Babs{ \sum_{S\in\mathscr{S}}\sum_{J:J\supsetneq S}B(\tilde P_S^\sigma f, 1_{S})\ave{\Delta_J^w g}_{J_S}^w}
   \lesssim\mathcal{H}^*\Norm{f}{L^2(\sigma)}\Norm{g}{L^2(w)}.
\end{equation*}

The remaining part of \eqref{eq:BbelowUnequal}, with $1_{J_S\setminus S}$ in place of $1_{J_S}$, is the \emph{tail form}, which we treat in the following section.
For reference, we record it here as
\begin{equation}\label{eq:Bbelow4Poisson}
  B_{\operatorname{tail}}(f,g):=\sum_{S\in\mathscr{S}}\sum_{J:J\supsetneq S}B(\tilde P_S^\sigma f, 1_{J_S\setminus S})\ave{\Delta_J^w g}_{J_S}^w
  =\sum_{S\in\mathscr{S}}B(\tilde P_S^\sigma f, \Phi_S^w g),
\end{equation}
where
\begin{equation}\label{eq:PhiS}
  \Phi_S^w g:=\sum_{J:J\supsetneq S}1_{J\setminus J_S}\ave{g}_{J}^w.
\end{equation}
Note that
\begin{equation*}
  \abs{\Phi_S^w g}\leq 1_{S^c}M_w g,\qquad M_w g:=\sup_{J\in\mathscr{D}}1_J\ave{\abs{g}}_J^w.
\end{equation*}

We summarize the considerations of this section as:

\begin{proposition}
For the bilinear forms defined above, we have
\begin{equation*}
  \Babs{B_{\operatorname{below}}(f,g)
     -B_{\operatorname{local}}(f,g)- B_{\operatorname{tail}}(f,g)}
     \lesssim\mathcal{H}^*\Norm{f}{L^2(\sigma)}\Norm{g}{L^2(w)}.
\end{equation*}
\end{proposition}

\section{The tail forms and Poisson integral estimates}

In this section, which contains some of the key novelties compared to Lacey et al.~\cite{LSSU:2wHilbert}, we estimate the tail forms $B_{\operatorname{tail}}(f,g)$ from above. Fundamental to their analysis is the connection with the Poisson integral, for which we use the notation and normalization
\begin{equation*}
  Q(h\ud w,J)
  :=\int_{\R}\frac{h(x)\ud w(x)}{\abs{J}^2+(x-c_J)^2}.
\end{equation*}
Notice that $\frac1\pi\abs{J}\cdot Q(h\ud w,J)$ is the usual Poisson extension of the signed measure $h\ud w$ at the point $(c_J,\abs{J})\in\R_+^2$. However, the operator $Q$ is more natural for our purposes: it is in this form that the Poisson integral appears in our computations, and this same form is also most amenable for its further analysis by the (dyadic) methods that we have in mind. Observe in particular that the second factor in the definition of the $A_2$ constant $[\sigma,w]_{A_2}^*$ is comparable to
\begin{equation*}
  Q(1_{J^c}\ud w,J)\eqsim\int_{J^c}\frac{\ud w(x)}{(x-c_J)^2}.
\end{equation*}

For each $S$, the projection $\tilde P^{\sigma}_S f$ is a sum of the differences $\Delta_I^\sigma f$ over good intervals such that $\pi I^{(r)}=S$, so in particular $I\Subset S$. Let us denote by $\mathscr{K}_S$ the maximal dyadic cubes $I$ with the just listed properties. Thus every $I$ participating in $\tilde P^{\sigma}_S f$ is contained in some $K\in\mathscr{K}_S$, and we can write
\begin{equation*}
  \tilde P^\sigma_S f
  =\sum_{K\in\mathscr{K}_S} 1_K   \tilde P^\sigma_S f
  =: \sum_{K\in\mathscr{K}_S} \tilde P^{\sigma}_{S,K} f,
\end{equation*}
where each $\tilde P^{\sigma}_{S,K}$ again a Haar projection.

The goodness of these intervals implies that
\begin{equation*}
   \dist(K,S^c)\geq\abs{K}^\gamma\abs{S}^{1-\gamma}\geq 2^{r(1-\gamma)}\abs{K}\qquad\forall K\in\mathscr{K}_S.
\end{equation*}
Since $\spt\Phi_S^{w}g\subseteq S^c$, there is some separation to help us.

We need the following ``monotonicity principle'' from Lacey et al.~\cite{LSSU:2wHilbert}:

\begin{proposition}[(Lacey--Sawyer--Shen--Uriarte-Tuero \cite{LSSU:2wHilbert})]\label{prop:monot}
Let $\mathscr{I}$ be a collection of dyadic intervals contained in some interval $J$, and let
$\spt g,\spt h\subseteq J^c$, where $\abs{g}\leq h$. Then
\begin{equation}\label{eq:monot1}
  \Babs{B\Big(\sum_{I\in\mathscr{I}}\Delta_I^\sigma f,g\Big)}
  \leq B\Big(\sum_{I\in\mathscr{I}}\epsilon_I\Delta_I^\sigma f,h\Big),
\end{equation}
where $\id$ is the function $\id(x)=x$, and $\epsilon_I=\sign\pair{\Delta_I^\sigma f}{\id}_{\sigma}\in\{-1,+1\}$.

We also have the estimate,
\begin{equation}\label{eq:monot2}
  B\Big(\sum_{I\in\mathscr{I}}\epsilon_I\Delta_I^\sigma f,h\Big)
  \gtrsim\Bpair{\sum_{I\in\mathscr{I}}\epsilon_I\Delta_I^\sigma f}{\id}_{\sigma}   Q(h\ud w,J),
\end{equation}
with ``$\eqsim$'' in place of ``$\gtrsim$'' if $\spt h\subseteq(3J)^c$.
\end{proposition}

\begin{proof}
We give a short proof for completeness.
Using $\int\Delta_I^\sigma f\ud\sigma=0$ to replace $1/(x-y)$ by $1/(x-y)-1/(x-c_I)$, we have
\begin{equation*}
  B\Big(\sum_{I\in\mathscr{I}}\Delta_I^\sigma f,g\Big)
  =\sum_{I\in\mathscr{I}}\iint\Delta_I^\sigma f(y)\frac{(y-c_I)}{(x-y)(x-c_I)}g(x)\ud\sigma(y)\ud w(x).
\end{equation*}
Note that $(x-y)(x-c_I)>0$ for $(x,y)\in J^c\times J$, since $x$ lies either to the left or to the right of $J$, while both $y,c_I\in J$. Also,  $(y-c_I)\Delta_I^\sigma f(y)$ has constant sign on its support, equal to the sign of its integral, which is
\begin{equation*}
  \sign\int (y-c_I)\Delta_I^\sigma f(y)\ud\sigma(y)
  =  \sign\int y\Delta_I^\sigma f(y)\ud\sigma(y)=\epsilon_I.
\end{equation*}
Thus, estimating up by bringing the absolute values inside and using $\abs{g(x)}\leq h(x)$, we arrive at the upper bound
\begin{equation*}
  \sum_{I\in\mathscr{I}}\iint\epsilon_I\Delta_I^\sigma f(y)\frac{(y-c_I)}{(x-y)(x-c_I)}h(x)\ud\sigma(y)\ud w(x)
  =B\Big(\sum_{I\in\mathscr{I}}\epsilon_I\Delta_I^\sigma f,h\Big),
\end{equation*}
which proves \eqref{eq:monot1}.

We can also derive a lower bound for this integral by observing that $\abs{x-y},\abs{x-c_I}\leq 2\abs{x-c_J}$, so that
\begin{equation*}
   B\Big(\sum_{I\in\mathscr{I}}\epsilon_I\Delta_I^\sigma f,h\Big)
   \geq\frac{1}{4}\Bpair{\sum_{I\in\mathscr{I}}\epsilon_I\Delta_I^\sigma f}{\id}_{\sigma}\int\frac{h(x)\ud w(x)}{\abs{x-c_J}^2}.
\end{equation*}
If $\spt h\subseteq (3J)^c$, we also have $\abs{x-c_J}\leq 2\abs{x-y}, 2\abs{x-c_I}$, so that we can reverse the previous estimate. This proves \eqref{eq:monot2} and its variant with ``$\eqsim$''.
\end{proof}

Denoting
\begin{equation*}
  \tilde{f}:=\sum_{I\in\mathscr{D}}\epsilon_I\Delta_I^\sigma f,\qquad\epsilon_I:=\sign\pair{\Delta_I^\sigma f}{\id}_{\sigma},
\end{equation*}
we thus have
\begin{equation}\label{eq:BkTailEst1}
\begin{split}
  \abs{B_{\operatorname{tail}}(f,g)}
  &\leq\sum_{S\in\mathscr{S}}\sum_{K\in\mathscr{K}_S}
       \abs{B(\tilde P_{S,K}^{\sigma}f,\Phi_S^{w}g)} \\
  &\lesssim\sum_{S\in\mathscr{S}}\sum_{K\in\mathscr{K}_S}\pair{\tilde P_{S,K}^{\sigma}\tilde{f}}{\id}_\sigma
     Q(\abs{\Phi_S^{w}g}\ud w,K) \\
  &\leq\Norm{f}{L^2(\sigma)}\Big(\sum_{S\in\mathscr{S}}\sum_{K\in\mathscr{K}_S}\Norm{\tilde P_{S,K}^{\sigma}\id }{L^2(\sigma)}^2
     Q(\abs{\Phi_S^{w}g}\ud w,K)^2\Big)^{1/2},
\end{split}
\end{equation}
using Cauchy--Schwarz, the orthogonality of the projections $\tilde P_{S,K}^{\sigma}$, and the fact that $\Norm{\tilde f}{L^2(\sigma)}=\Norm{f}{L^2(\sigma)}$.

For the subsequent analysis, we want to replace $Q(\cdot,K)$ by a more dyadic object. Let $\mathscr{D}$ be our underlying dyadic system, and recall that the collection of tripled intervals $\{3I:I\in\mathscr{D}\}$ can be partitioned into three subcollections $\mathscr{D}^u$, $u=0,1,2$, each of which is a new dyadic system (in the sense of nesting and covering properties), except that the side-lengths of the intervals are of the form $3\cdot 2^k$.

For an interval $K$ and $u=0,1,2$, we denote by $I^u(K)$ the unique (if it exists) $I\in\mathscr{D}^u$ such that $3K\subset I$ and $9\abs{K}<\abs{I}\leq 18\abs{K}$. (If, as in our applications, $K\in\mathscr{D}$ itself is dyadic, the last condition means that $\abs{I}=12\abs{K}$.) For definiteness, let $I^u(K):=\varnothing$ if such an interval does not exist.

We need the following geometric result, which is essentially \cite[Lemma~2.5]{HLP}, and can be proven in the same way.

\begin{lemma}[(cf.~\cite{HLP}, Lemma~2.5)]
For every $K$ and $j\geq 1$, we have at least two values of $u$ such that $I^u(K)\supset 3K$, and at least one value of $u$ such that in addition $(I^u(K))^{(j)}\supset 3\cdot 2^j K$.
\end{lemma}

This leads to the following dyadic approximation of the Poisson integral:

\begin{proposition}
For $h\geq 0$, we have
\begin{equation*}
  Q(h,K)
  \eqsim\sum_{u:I^u(K)\neq\varnothing}Q^u(h,I^u(K)),
\end{equation*}
where
\begin{equation*}
  Q^u(h,I):=\sum_{j=0}^\infty\frac{1}{\abs{I^{(j)}}^2}\int_{I^{(j)}}h,
\end{equation*}
and the ancestors $I^{(j)}$ are taken within the system $\mathscr{D}^u$.
\end{proposition}

\begin{proof}
It is easy to see that
\begin{equation*}
  Q(h,K)\eqsim\sum_{j=0}^\infty \frac{1}{\abs{2^j K}^2}\int_{2^j K}h.
\end{equation*}
If $I=I^u(K)$, we have $I^{(j)}\subset 2^{6+j}K$, so that $\int_{I^{(j)}}h\leq\int_{2^{6+j}K}h$ while $\abs{{I}^{(j)}}\eqsim\abs{2^{6+j}K}$, from which $Q^u(h,I^u(K))\lesssim Q(h,K)$ immediately follows. On the other hand, for each $j$, let $u(j)$ be a value for which the conclusion of the Lemma are valid, so that $\int_{(I^{u(j)}(K))^{(j)}}h\geq \int_{2^j K}h$ while $\abs{ (I^{u(j)}(K))^{(j)} } \eqsim \abs{ 2^j K }$. Hence
\begin{equation*}
\begin{split}
  \sum_{u:I^u(K)\neq\varnothing}Q^u(h,I^u(K))
  &\geq\sum_{j=0}^\infty\frac{1}{\abs{ (I^{u(j)}(K))^{(j)} }^2 }\int_{(I^{u(j)}(K))^{(j)}} h \\
  &\gtrsim\sum_{j=0}^\infty\frac{1}{\abs{ 2^j K}^2}\int_{2^j K} h
  \eqsim Q(h,K).\qedhere
\end{split}
\end{equation*}
\end{proof}

Defining $Q^u(h,\varnothing):=0$, we can simply write
\begin{equation*}
  Q(h,K)\eqsim\sum_u Q^u(h,I^u(K))
\end{equation*}
and then \eqref{eq:BkTailEst1} implies that
\begin{equation*}
\begin{split}
  &\abs{B_{\operatorname{tail}}(f,g)}^2 \\
  &\lesssim\Norm{f}{L^2(\sigma)}^2\sum_{u\in\{0,1,2\}}\sum_{S\in\mathscr{S}}\sum_{K\in\mathscr{K}_S}\Norm{\tilde P_{S,K}^{\sigma}\id }{L^2(\sigma)}^2
     Q^u(\abs{\Phi_S^{w}g}\ud w,I^u(K))^2.
\end{split}
\end{equation*}
Taking into account that
\begin{equation*}
  \abs{\Phi_S^{w}g}
  \lesssim 1_{S^c}M_w g
  \leq 1_{(I^u(K))^c}M_w g,
\end{equation*}
where we recalled that $\dist(K,S^c)\geq 2^{(1-\gamma)r}\abs{K}$ and observed that $I^u(K)\subseteq 30K$, we may continue with
\begin{equation*}
\begin{split}
  \abs{B_{\operatorname{tail}}(f,g)}^2
  &\lesssim\Norm{f}{L^2(\sigma)}^2\sum_{u\in\{0,1,2\}}\sum_{I\in\mathscr{D}^u}\mu_I\cdot
     Q^u(1_{I^c}M_w g\ud w,I)^2, \\
  \mu_I & :=  \sum_{S\in\mathscr{S}}\sum_{\substack{K\in\mathscr{K}_S \\ I^u(K)=I }}\Norm{\tilde P_{S,K}^{\sigma}\id }{L^2(\sigma)}^2
\end{split}
\end{equation*}

To complete the estimate of the tail forms, we would like to have the bound
\begin{equation}\label{eq:tailNeedToComplete}
  \sum_{I\in\mathscr{D}^u}\mu_I\cdot
     Q^u(1_{I^c} h\ud w,I)^2
    \lesssim \mathcal{Q}^2 \Norm{h}{L^2(w)}^2\qquad\forall h\in L^2(w),
\end{equation}
with appropriate control of $\mathcal{Q}$ in terms of the assumptions; namely, we just apply this to $h=M_w g$ and use the universal bound for the dyadic maximal operator, $\Norm{M_w g}{L^2(w)}\leq 2\Norm{g}{L^2(w)}$.

Proving the Poisson integral estimate \eqref{eq:tailNeedToComplete} will occupy our efforts for the rest of this section.

\subsection{Positive dyadic operators}

As a tool for estimating the Poisson integrals from the previous considerations, we develop a two-weight theory for a class of positive dyadic operators that are slightly more complicated than those considered by Nazarov--Treil--Volberg \cite{NTV:bilinear} and Lacey--Sawyer--Uriarte-Tuero \cite{LSU:positive}. Our approach to these operators is an elaboration of the argument from \cite{Hytonen:A2survey}.

Let $\mathscr{D}$ be a system of dyadic cubes in some $\R^d$, $d\geq 1$, and let $\sigma,w$ be two locally finite measures. (We shall essentially need the case $d=2$, but prefer to write the argument with a generic $d$, since there is no added difficulty.) To every $Q\in\mathscr{D}$, let a nonnegative number $\lambda_Q\geq 0$ and two distinguished child-cubes $Q_+\neq Q_-$ be associated. We then consider the bilinear form
\begin{equation*}
   \Lambda(f,g):=\sum_{Q\in\mathscr{D}}\lambda_Q\int_{Q_+} f\ud\sigma\int_{Q_-}g\ud w
\end{equation*}
and the sub-forms
\begin{equation*}
  \Lambda_R(f,g):=\sum_{Q:Q\subseteq R}\lambda_Q\int_{Q_+} f\ud\sigma\int_{Q_-}g\ud w
\end{equation*}

The previous studies \cite{Hytonen:A2survey,LSU:positive,NTV:bilinear} considered a version with $Q_-=Q_+=Q$. The present form has the slight additional complication of interactions between the distinct adjacent cubes $Q_-$ and $Q_+$. Nevertheless, the theorem stays almost the same.

\begin{theorem}\label{thm:posDyad}
Let $1<p<\infty$. We have
\begin{equation*}
   \Norm{\Lambda}{}:=\sup_{\substack{\Norm{f}{L^p(\sigma)}=1 \\ \Norm{g}{L^{p'}(w)}=1}}\abs{\Lambda(f,g)}<\infty
\end{equation*}
if and only if
\begin{equation*}
\begin{split}
  \mathcal{U} &:=\sup_{Q\in\mathscr{D}}\lambda_Q\sigma(Q_+)^{1/p'}w(Q_-)^{1/p}<\infty, \\
  \mathcal{T} &:=\sup_{\substack{Q\in\mathscr{D} \\ \Norm{g}{L^{p'}(w)}=1}}\sigma(Q)^{-1/p}\abs{\Lambda_Q(1,g)}<\infty, \\
  \mathcal{T}^* &:=\sup_{\substack{Q\in\mathscr{D} \\ \Norm{f}{L^{p}(\sigma)}=1}}w(Q)^{-1/p'}\abs{\Lambda_Q(f,1)}<\infty,
\end{split}
\end{equation*}
and moreover $\Norm{\Lambda}{}\eqsim\mathcal{U}+\mathcal{T}+\mathcal{T}^*$.
\end{theorem}

\begin{proof}
The nontrivial part is to check that $\Norm{\Lambda}{}\lesssim\mathcal{U}+\mathcal{T}+\mathcal{T}^*$, where we may further restrict the considerations to nonnegative functions and summations over cubes contained in some large initial cube $Q^\circ$. We define the usual principal cubes $\mathscr{F}=\bigcup_{k=0}^\infty\mathscr{F}_k$ for $(f,\sigma)$ by $\mathscr{F}_0:=\{Q^\circ\}$ and
\begin{equation*}
  \mathscr{F}_{k+1}:=\bigcup_{F\in\mathscr{F}_k}\operatorname{ch}_{\mathscr{F}}(F),\qquad
  \operatorname{ch}_{\mathscr{F}}(F):=\{F'\subsetneq F\text{ maximal}: \ave{f}_{F'}^\sigma>2\ave{f}_{F}^\sigma\}.
\end{equation*}
We also use the notation
\begin{equation*}
  \pi_{\mathscr{F}}(Q):=\min\{F\in\mathscr{F}:F\supseteq Q\},\qquad
  E_{\mathscr{F}}(F):=F\setminus\bigcup_{F'\in\operatorname{ch}_{\mathscr{F}}F}F',
\end{equation*}
where the latter sets are pairwise disjoint and $\sigma(E_{\mathscr{F}}(F))\geq\tfrac12\sigma(F)$. The collection $\mathscr{G}$ for $(g,w)$, and the derived notions, are defined analogously. 

The idea is to organize the cubes $Q\in\mathscr{D}$ under the appropriate $F\in\mathscr{F}$ and $G\in\mathscr{G}$. We first handle a certain ``diagonal'' case:
\begin{equation*}
\begin{split}
  \sum_{\substack{Q\in\mathscr{D}:\\ Q_+\in\mathscr{F} \\ Q_-\in\mathscr{G}}} & \lambda_Q\int_{Q_+}f\ud\sigma\int_{Q_-}g\ud w \\
  &\leq\mathcal{U}\sum_{\substack{Q\in\mathscr{D}:\\ Q_+\in\mathscr{F} \\ Q_-\in\mathscr{G}}}\ave{f}_{Q_+}^\sigma\ave{g}_{Q_-}^w
       \sigma(Q_+)^{1/p}w(Q_-)^{1/p'} \\
      & \leq \mathcal{U}\Big(\sum_{\substack{Q\in\mathscr{D}:\\ Q_+\in\mathscr{F}}}(\ave{f}_{Q_+}^\sigma)^p\sigma(Q_+)\Big)^{1/p}
        \Big(\sum_{\substack{Q\in\mathscr{D}:\\ Q_-\in\mathscr{G}}}(\ave{g}_{Q_-}^w)^{p'} w(Q_-)\Big)^{1/p'} \\
      & \leq \mathcal{U}\Big(\sum_{F\in\mathscr{F}}(\ave{f}_{F}^\sigma)^p\sigma(F)\Big)^{1/p}
        \Big(\sum_{G\in\mathscr{G}}(\ave{g}_{G}^w)^{p'} w(G)\Big)^{1/p'} \\
     &   \lesssim \mathcal{U}\Norm{f}{L^p(\sigma)}\Norm{g}{L^{p'}(w)}.
\end{split}
\end{equation*}

We are left with the part of the sum where at least one of $Q_+\notin\mathscr{F}$ or $Q_-\notin\mathscr{G}$ holds. For instance the first case means that $\pi_{\mathscr{F}}Q_+=\pi_{\mathscr{F}}Q\supseteq Q\supsetneq Q_-$. Since also $\pi_{\mathscr{G}}Q_-\supseteq Q_-$, we conclude that $\pi_{\mathscr{F}}Q_-\cap\pi_{\mathscr{G}}Q_+\neq\varnothing$, and therefore one of these cubes is contained in the other one. The same conclusion of course follows by symmetry if $Q_-\notin\mathscr{G}$. Thus, the remaining part of the sum may be reorganized as
\begin{equation*}
  \sum_{\substack{F\in\mathscr{F}\\ G\in\mathscr{G}\\ F\cap G\neq\varnothing}}\sum_{\substack{Q\in\mathscr{D}\\ \pi_{\mathscr{F}}Q_+=F\\ \pi_{\mathscr{G}}Q_- = G}}
  \leq\Big(\sum_{F\in\mathscr{F}}\sum_{\substack{G\in\mathscr{G}\\ G\subseteq F}}+\sum_{G\in\mathscr{G}}\sum_{\substack{F\in\mathscr{F} \\ F\subseteq G}}\Big)
   \sum_{\substack{Q\in\mathscr{D}\\ \pi_{\mathscr{F}}Q_+=F\\ \pi_{\mathscr{G}}Q_- = G}},
\end{equation*}
where we double-counted the diagonal $F=G$ to keep the two parts symmetric. By symmetry, we concentrate on the first half.

For fixed $F$ and $G$, consider a cube $Q$ in the innermost sum. Note that also $\pi_{\mathscr{F}}Q=\pi_{\mathscr{F}}Q_+=F$, for if not, then $F=Q_+$, which leads to the contradiction that $Q_-\subseteq G\subseteq F=Q_+$.

Now consider some $F'\in\operatorname{ch}_{\mathscr{F}}F$. We cannot have $Q_-\subsetneq F'$, for this would imply that $Q\subseteq F'$, contradicting with $\pi_{\mathscr{F}}Q=F$. Thus $Q_-\cap F'\neq\varnothing$ only if $F'\subseteq Q_-\subseteq G\subseteq F$, where $G\in\mathscr{G}$. This leads to the formula
\begin{equation*}
  \int_{Q_-}g\ud w
  =\int_{Q_-}g_F\ud w,\qquad
  g_F:= 1_{E_{\mathscr{F}}(F)}g+\sum_{F'\in\operatorname{ch}_{\mathscr{F}}^* F}\ave{g}_{F'}^w 1_{F'},
\end{equation*}
where
\begin{equation*}
  \operatorname{ch}_{F}^* F
  :=\{F'\in\operatorname{ch}_{\mathscr{F}}F: \pi_{\mathscr{G}}F'\subseteq F\}.
\end{equation*}
Observing also that $\ave{f}_{Q_+}^\sigma\leq 2\ave{f}_F^\sigma$ for all cubes $Q$ under consideration, we conclude that
\begin{equation*}
\begin{split}
  \sum_{\substack{G\in\mathscr{G}\\ G\subseteq F}}
   \sum_{\substack{Q\in\mathscr{D}\\ \pi_{\mathscr{F}}Q_+=F\\ \pi_{\mathscr{G}}Q_- = G}}
   \lambda_Q\int_{Q_+}f\ud\sigma\int_{Q_-}g\ud w 
  &\leq 2\ave{f}_{F}^\sigma \sum_{Q:Q\subseteq F}
   \lambda_Q\sigma(Q_+)\int_{Q_-}g_F\ud w \\
   &= 2\ave{f}_{F}^\sigma\Lambda_F(1,g_F)
   \leq 2\mathcal{T}\ave{f}_F^\sigma\sigma(F)^{1/p}\Norm{g_F}{L^{p'}(w)}.
\end{split}
\end{equation*}

Summing over $F\in\mathscr{F}$ and using H\"older's inequality, the part involving $(f,\sigma)$ is bounded by $\Norm{f}{L^p(\sigma)}$, and we are left with
\begin{equation*}
  \sum_{F\in\mathscr{F}}\Norm{g_F}{L^{p'}(w)}^{p'}
  =\sum_{F\in\mathscr{F}}\Norm{1_{E_{\mathscr{F}}(F)}g}{L^{p'}(w)}^{p'}
    +\sum_{F\in\mathscr{F}}\sum_{F'\in\operatorname{ch}_{\mathscr{F}}^* F}(\ave{g}_{F'}^w)^{p'}w(F').
\end{equation*}
The first sum is immediately estimated by $\Norm{g}{L^{p'}(w)}^{p'}$, since the sets $E_{\mathscr{F}}(F)$ are pairwise disjoint.

To proceed with the double sum, we wish to reorganize the inner summation with respect to $G:=\pi_{\mathscr{G}}F'$. From the definition of $\operatorname{ch}_{\mathscr{F}}^*(F)$, we find that either $G=F'$, or else $F'\subsetneq G\subseteq F$, whence $\pi_{\mathscr{F}}G=F$. Thus
\begin{equation*}
\begin{split}
  \sum_{F'\in\operatorname{ch}_{\mathscr{F}}^* F}(\ave{g}_{F'}^w)^{p'}w(F')
  &=\sum_{\substack{G\in\mathscr{G}:\\ \pi_{\mathscr{F}}G=F\text{ or}\\ G\in\operatorname{ch}_{\mathscr{F}}(F)}}
      \sum_{\substack{F'\in\operatorname{ch}_{\mathscr{F}} F\\ \pi_{\mathscr{G}}F'=G}}(\ave{g}_{F'}^w)^{p'}w(F') \\
   &\leq\sum_{\substack{G\in\mathscr{G}:\\ \pi_{\mathscr{F}}G=F\text{ or}\\ G\in\operatorname{ch}_{\mathscr{F}}(F)}}2^{p'}(\ave{g}_G^w)^{p'}
      \sum_{\substack{F'\in\operatorname{ch}_{\mathscr{F}} F\\ \pi_{\mathscr{G}}F'=G}}w(F') \\
&\leq\sum_{\substack{G\in\mathscr{G}:\\ \pi_{\mathscr{F}}G=F\text{ or}\\ G\in\operatorname{ch}_{\mathscr{F}}(F)}}2^{p'}(\ave{g}_G^w)^{p'}w(G).
\end{split}
\end{equation*}
Summing over $F\in\mathscr{F}$,
denoting by $\pi_{\mathscr{F}}^1 G$ is the smallest $F'\in\mathscr F$ with $\pi_{\mathscr{F}} G\subsetneq F'$,
we obtain
\begin{equation*}
  \sum_{F\in\mathscr{F}}\sum_{\substack{G\in\mathscr{G}:\\ \pi_{\mathscr{F}}G=F\text{ or}\\ G\in\operatorname{ch}_{\mathscr{F}}(F)}}2^{p'}(\ave{g}_G^w)^{p'}w(G)
  \leq \sum_{G\in\mathscr{G}}2^{p'}(\ave{g}_G^w)^{p'}w(G)\sum_{\substack{F\in\mathscr{F}\\ F=\pi_{\mathscr{F}}G\text{ or}\\ F=\pi_{\mathscr{F}}^1 G }}1.
\end{equation*}
The inner sum is bounded by $2$, and the rest is estimated by $\Norm{g}{L^{p'}(w)}^{p'}$, as before.
\end{proof}

\subsection{Testing conditions for the Poisson integral with holes}

We return to our desired tail form estimate \eqref{eq:tailNeedToComplete}. By `holes', we refer to the presence of the indicator $1_{I^c}$ in this estimate: the integral has a hole over the interval $I$. Dualizing, \eqref{eq:tailNeedToComplete} is equivalent to
\begin{equation*}
  \sum_{I\in\mathscr{D}^u}\mu_I\cdot Q^u(1_{I^c}h\ud w,I)\cdot\phi_I
  \leq\mathcal{Q}\Norm{h}{L^2(w)}\Big(\sum_{I\in\mathscr{D}^u}\phi_I^2\mu_I\Big)^{1/2}
\end{equation*}
for all sequences $\{\phi_I\}_{I\in\mathscr{D}^u}\in\ell^2(\{\mu_I\}_{I\in\mathscr{D}^u})$.
Let us also observe that
\begin{equation*}
\begin{split}
  Q^u(1_{I^c}h\ud w,I)
  &=\sum_{J\supsetneq I}\frac{1}{\abs{J}^2}\int_{J\setminus I}h\ud w
  =\sum_{J\supsetneq I}\frac{1}{\abs{J}^2}\sum_{\substack{J'\supseteq I \\ J'\subsetneq J}}\int_{(J')^{(1)}\setminus J'}h\ud w \\
  &=\sum_{J'\supseteq I}\int_{(J')^{(1)}\setminus J'}h\ud w \sum_{J\supsetneq J'}\frac{1}{\abs{J}^2}
  \eqsim\sum_{J'\supseteq I}\int_{(J')^{(1)}\setminus J'}h\ud w \frac{1}{\abs{J'}^2}.
\end{split}
\end{equation*}
We define a measure $\mu$ on $\R^2$ by
\begin{equation}\label{eq:muDef}
  \mu:=\sum_{I\in\mathscr{D}^u}\mu_I\cdot\delta_{c(W(I))},
\end{equation}
where
\begin{equation*}
  W(I):=I\times[\frac12\abs{I},\abs{I}),\qquad c(W(I))=(c(I),\frac32\abs{I})
\end{equation*}
are the Whitney region associated with $I$, and its centre. Let us also define the Carleson box
\begin{equation*}
  \hat J:=J\times[0,\abs{J})=(J\times\{0\})\cup\bigcup_{I\subseteq J}W(I),
\end{equation*}
with $\abs{\hat J}=\abs{J}^2$, where $\abs{\ }$ liberally refers to the Lebesgue measure of appropriate dimension.
We write $\phi$ for a function on $\R^2$ that takes the value $\phi_I$ at $c(W(I))$ for each $I$; its other values are immaterial.

Thus
\begin{equation}\label{eq:PoissonRewrite}
\begin{split}
  \sum_{I\in\mathscr{D}^u}\mu_I\cdot Q^u(1_{I^c}h\ud w,I)\cdot\phi_I
  &\eqsim\sum_{I\in\mathscr{D}^u}\mu_I\cdot\sum_{J:J\supseteq I}\frac{1}{\abs{\hat J}}\int_{J^{(1)}\setminus J}h\ud w\cdot\phi_I \\
  &=\sum_{J\in\mathscr{D}^u}\frac{1}{\abs{\hat J}}\int_{J^{(1)}\setminus J}h\ud w\sum_{I:I\subseteq J}\mu_I\phi_I \\
  &=\sum_{J\in\mathscr{D}^u}\frac{1}{\abs{\hat J}}\int_{J^{(1)}\setminus J}h\ud w\int_{\hat{J}}\phi\ud\mu.
\end{split}
\end{equation}
We next check that this is a sum of  two operators of the form considered in Theorem~\ref{thm:posDyad}. Namely, reorganizing the summation with respect to $I:=J^{(1)}$, we have
\begin{equation*}
\begin{split}
  \sum_{J\in\mathscr{D}^u}\frac{1}{\abs{\hat J}}\int_{J^{(1)}\setminus J}h\ud w\int_{\hat{J}}\phi\ud\mu
  &\eqsim\sum_{\substack{i,j\in\{\operatorname{left},\operatorname{right}\} \\ i\neq j}}\sum_{I\in\mathscr{D}^u}\frac{1}{\abs{\hat I}}
      \int_{I_i}h\ud w\int_{\hat I_j}\phi\ud\mu \\
  &=\sum_{\substack{i,j\in\{\operatorname{left},\operatorname{right}\} \\ i\neq j}}\sum_{R\in\mathscr{D}^{u,2}}\lambda_R
      \int_{R_i}h\,\ud(w\times\delta_0)\int_{R_j}\phi\ud\mu,
\end{split}
\end{equation*}
where
\begin{equation*}
  \mathscr{D}^{u,2}:=\{R=I\times [m\abs{I},(m+1)\abs{I}): I\in\mathscr{D}^u,m\in\Z\}
\end{equation*}
is a collection of dyadic cubes in $\R^2$; the coefficients $\lambda_R$ vanish unless $R=\hat I=I\times[0,\abs{I})$, in which case we have $\lambda_R:=1/\abs{\hat I}$ and the two distinguished subcubes $R_k:=\hat I_k=I_k\times[0,\abs{I_k})$ for $k\in\{\operatorname{left},\operatorname{right}\}$. 

From Theorem~\ref{thm:posDyad} we thus conclude the following:

\begin{corollary}\label{cor:tailToComplete}
The inequality \eqref{eq:tailNeedToComplete} holds if and only if
\begin{equation*}
\begin{split}
  \mathcal{U} &:=\sup_{J\in\mathscr{D}^u}\frac{1}{\abs{\hat J}}w(J^{(1)}\setminus J)^{1/2}\mu(\hat J)^{1/2}<\infty,\\
  \mathcal{T} &:=\sup_{J\in\mathscr{D}^u}\frac{1}{w(J)^{1/2}}\BNorm{\sum_{I:I\subsetneq J}\frac{1_{\hat I}}{\abs{\hat I}}w(I^{(1)}\setminus I)}{L^2(\mu)}<\infty,\\
  \mathcal{T}^* &:=\sup_{J\in\mathscr{D}^u}\frac{1}{\mu(\hat J)^{1/2}}\BNorm{\sum_{I:I\subsetneq J} 
      1_{I^{(1)}\setminus I}\frac{\mu(\hat I)}{\abs{\hat I}} }{L^2(w)}<\infty.
\end{split}
\end{equation*}
Moreover, $\mathcal{Q}\eqsim\mathcal{U}+\mathcal{T}+\mathcal{T}^*$.
\end{corollary}

The rest of this section is devoted to the estimation of the testing constants from Corollary~\ref{cor:tailToComplete}. Recall that the measure $\mu$ was defined in \eqref{eq:muDef}, so in particular
\begin{equation*}
  \mu(\hat J)
  =\sum_{I:I\subseteq J}\mu_I,\qquad
    \mu_I=\sum_{S\in\mathscr{S}}\sum_{\substack{K\in\mathscr{K}_S \\ I^u(K)=I}}\Norm{\tilde P_{S,K}^{\sigma}\operatorname{id}}{L^2(\sigma)}^2.
\end{equation*}
Thus
\begin{equation*}
  \mu(\hat J)
  =\sum_{S\in\mathscr{S}}\sum_{\substack{K\in\mathscr{K}_S \\ K\subseteq I^u(K)\subseteq J}}\Norm{\tilde P_{S,K}^{\sigma}\operatorname{id}}{L^2(\sigma)}^2.
\end{equation*}
(By inspection of the definition of $I^u(K)$, one can check that ``$K\subseteq I^u(K)$'' simply means that $I^u(K)$ is some interval $I$, rather than $\varnothing$.)
Here the projections $\tilde P^{\sigma}_{S,K}$ are pairwise orthogonal in $(S,K)$. Under the restriction that $I^u(K)\subseteq J$, they all project onto Haar functions $h_L^\sigma$ supported inside $J$. Since $\int h^\sigma_L\ud\sigma=0$, we may replace $\id$ by $1_J(\id-c_J)$, to find that
\begin{equation}\label{eq:muJsimple}
  \mu_J\leq\mu(\hat{J})\leq\abs{J}^2\sigma(J).
\end{equation}

These simple bounds are enough for much of our needs.
In particular, it immediately follows that
\begin{equation*}
  \frac{1}{\abs{\hat J}}w(J^{(1)}\setminus J)^{1/2}\mu(\hat J)^{1/2}
  \leq \frac{1}{\abs{J}}w(J^{(1)}\setminus J)^{1/2}\sigma(J)^{1/2}\leq[\sigma,w]_{A_2},
\end{equation*}
and thus
\begin{equation*}
  \mathcal{U}\leq[\sigma,w]_{A_2}.
\end{equation*}

\begin{remark}
Corollary~\ref{cor:tailToComplete} is our characterization for the two-weight boundedness of the dyadic ``Poisson integral with holes'', appearing in \eqref{eq:tailNeedToComplete}. We use this dyadic characterization to derive and verify a \emph{sufficient} condition for the two-weight inequality of a classical Poisson integral with holes, as in \eqref{eq:BkTailEst1}; however, we make no attempt towards actually characterizing this latter estimate. Indeed, in passing between the dyadic and the classical settings, we lose the precise information on the size of the holes. However, in the absence of holes, there is precise characterization for the two-weight inequality of the classical Poisson integral, due to Sawyer~\cite{Sawyer:2wFractional}, and this is one of the key ingredients behind the work of Lacey et al.~\cite{LSSU:2wHilbert}. Finding a substitute result for the case of holes was identified as a key to the removal of the ``no common point masses'' assumption by Lacey~\cite{Lacey:2wHilbertPrimer}.
\end{remark}

\subsection{Estimate for the testing constant $\mathcal{T}^*$}
This is more involved than the estimate for $\mathcal{U}$ above.
Let us fix an interval $J$ and expand
\begin{equation*}
\begin{split}
  \BNorm{\sum_{I:I\subsetneq J}1_{I^{(1)}\setminus I}\frac{\mu(\hat I)}{\abs{\hat I}}}{L^2(w)}^2
  &=\int\sum_{I,L\subsetneq J}\frac{1_{I^{(1)}\setminus I}1_{L^{(1)}\setminus L}}{\abs{\hat I}\abs{\hat L}}\sum_{H:H\subseteq I}\mu_I\sum_{K:K\subseteq L}\mu_K\ud w \\
  &\leq\sum_{H,K\subseteq J}\mu_H\mu_K\int\sum_{I:I\supseteq H}\frac{1_{I^{(1)}\setminus I}}{\abs{\hat I}}
     \sum_{L:L\supseteq K}\frac{1_{L^{(1)}\setminus L}}{\abs{\hat L}}\ud w \\
   &=:\sum_{H,K\subseteq J}\mu_H\mu_K\int q_H q_K \ud w,\qquad q_H:=\sum_{I:I\supseteq H}\frac{1_{I^{(1)}\setminus I}}{\abs{\hat I}}.
\end{split}
\end{equation*}
Now if $a_{HK}=a_{KH}\geq 0$, we clearly have
\begin{equation*}
  \sum_{H,K\subseteq J}a_{HK}\leq 2\sum_{\substack{H,K\subseteq J\\ \abs{K}\leq\abs{H}}}a_{HK},
\end{equation*}
and thus
\begin{equation*}
\begin{split}
  \BNorm{\sum_{I:I\subsetneq J}1_{I^{(1)}\setminus I}\frac{\mu(\hat I)}{\abs{\hat I}}}{L^2(w)}^2
  &\leq 2\sum_{H:H\subseteq J}\mu_H\int q_H\sum_{K: \abs{K}\leq\abs{H}}\mu_K q_K\ud w \\
  &\leq 2\mu(\hat J)\sup_H \int q_H\sum_{K: \abs{K}\leq\abs{H}}\mu_K q_K\ud w.
\end{split}
\end{equation*}
It remains to estimate the last supremum, which is the content of the following lemma:

\begin{lemma}
For nonnegative coefficients $\mu_J$ with $\mu_J\leq\abs{J}^2\sigma(J)$, we have
\begin{equation*}
  \int q_J\sum_{I:\abs{I}\leq\abs{J}}\mu_I q_I\ud w
  \lesssim ([w,\sigma]_{A_2}^*)^2.
\end{equation*}
\end{lemma}

\begin{proof}
We first write out the sum
\begin{equation}\label{eq:sumMuIgI}
  \sum_{I:\abs{I}\leq\abs{J}}\mu_I q_I
  =\sum_{I:\abs{I}\leq\abs{J}}\mu_I \sum_{K:K\supseteq I} 1_{K^{(1)}\setminus K}\abs{K}^{-2}
  =\sum_K\abs{K}^{-2}1_{K^{(1)}\setminus K} \sum_{\substack{I:I\subseteq K\\ \abs{I}\leq\abs{J}}}\mu_I,
\end{equation}
where the inner sum satisfies
\begin{equation*}
   \sum_{\substack{I:I\subseteq K\\ \abs{I}\leq\abs{J}}}\mu_I
  \leq\sum_{i=0}^\infty\sum_{\substack{I:I\subseteq K\\ \abs{I}=2^{-i}(\abs{J}\wedge\abs{K}) }}\abs{I}^2\sigma(I)
  =\sum_{i=0}^\infty 2^{-2i}(\abs{J}\wedge\abs{K})^2\sigma(K),
\end{equation*}
and the geometric series is summable.

If we expand $q_J=\sum_{L:L\supseteq J}\abs{L}^{-2}1_{L^{(1)}\setminus L}$ and multiply this by \eqref{eq:sumMuIgI}, we end up with indicators of the sets
\begin{equation*}
   (K^{(1)}\setminus K)\cap(L^{(1)}\setminus L)
   =\begin{cases}
        K^{(1)}\setminus K, & \text{if }K=L\text{ or }K\subsetneq L^{(1)}\setminus L, \\
        L^{(1)}\setminus L, & \text{if }L\subsetneq K^{(1)}\setminus K, \\
        \varnothing, & \text{otherwise}.
      \end{cases}  
\end{equation*}
Note that $L\subsetneq K^{(1)}\setminus K$ if and only if $K=H^{(1)}\setminus H$ for some $H\supsetneq L$.
Thus
\begin{equation*}
\begin{split}
  q_J\sum_{I:\abs{I}\leq\abs{J}}\mu_I q_I
  & \lesssim \sum_{K:K\supseteq J}\frac{1_{K^{(1)}\setminus K}}{\abs{K}^4}\abs{J}^2\sigma(K) \\
  &\qquad+\sum_{L:L\supseteq J}\sum_{K:K\subsetneq L^{(1)}\setminus L}
     \frac{1_{K^{(1)}\setminus K}}{\abs{L}^2\abs{K}^2}(\abs{J}\wedge\abs{K})^2\sigma(K) \\
  &\qquad+\sum_{L:L\supseteq J}\sum_{H:H\supseteq L^{(1)}}
     \frac{1_{L^{(1)}\setminus L}}{\abs{L}^2\abs{H}^2}\abs{J}^2\sigma(H^{(1)}\setminus H)
     =:I+II+III.
\end{split}
\end{equation*}
Using $w(K^{(1)}\setminus K)\sigma(K)\leq[\sigma,w]_{A_2}^2\abs{K}^2$, we deduce that
\begin{equation*}
  \int I\,\ud w
  \lesssim[\sigma,w]_{A_2}^2\sum_{K:K\supseteq J}\frac{\abs{J}^2}{\abs{K}^2}\lesssim[\sigma,w]_{A_2}^2
\end{equation*}
and
\begin{equation*}
\begin{split}
  \int II\,\ud w
  &\lesssim[\sigma,w]_{A_2}^2\sum_{L:L\supseteq J}\Big(
  \sum_{\substack{K:K\subseteq L^{(1)}\setminus L \\ \abs{K}\leq\abs{J} }}\frac{\abs{K}^2}{\abs{L}^2}
     +\sum_{\substack{K:K\subseteq L^{(1)}\setminus L \\ \abs{K}>\abs{J} }}\frac{\abs{J}^2}{\abs{L}^2}\Big) \\
   &\lesssim[\sigma,w]_{A_2}^2\sum_{L:L\supseteq J}\Big(\frac{\abs{J}\abs{L}}{\abs{L}^2}+\frac{\abs{L}}{\abs{J}}\frac{\abs{J}^2}{\abs{L}^2}\Big)
   \lesssim[\sigma,w]_{A_2}^2\sum_{L:L\supseteq J}\frac{\abs{J}}{\abs{L}}\lesssim[\sigma,w]_{A_2}^2,
\end{split}
\end{equation*}
where, passing from the first line to the second, we essentially summed up a geometric series in the first term, and in the second one, we just estimated the number of terms in the sum.

In the final term, we need to invoke the Poisson $A_2$ condition, but then it is also immediate:
\begin{equation*}
\begin{split}
  \int III\,\ud w
  &\lesssim\sum_{L:L\supseteq J}\frac{\abs{J}^2}{\abs{L}^2}w(L^{(1)})\sum_{H:H\supseteq L^{(1)}}\frac{\sigma(H^{(1)}\setminus H)}{\abs{H}^2} \\
  &\lesssim\sum_{L:L\supseteq J}\frac{\abs{J}^2}{\abs{L}^2}([w,\sigma]_{A_2}^*)^2\lesssim([w,\sigma]_{A_2}^*)^2. \qedhere
\end{split}
\end{equation*}
\end{proof}

A combination of the Lemma with the preceding considerations in this subsection shows that
\begin{equation*}
  \mathcal{T}^*\lesssim[w,\sigma]_{A_2}^*.
\end{equation*}

\subsection{Estimate for the testing constant $\mathcal{T}$}
This is a rather delicate part of the argument, already in \cite{LSSU:2wHilbert}: although we have a positive operator to bound in the beginning, we carry out the estimate by using Proposition~\ref{prop:monot} in the ``reverse'' direction, so as to transform it back to a singular integral with cancellation. The advantage of this is to make the interval testing conditions available to us.

Reversing the steps in \eqref{eq:PoissonRewrite}, we find that $\mathcal{T}$ is also given as the best constant in
\begin{equation*}
  \Big(\sum_{I:I\subsetneq J}\mu_I\cdot Q^u(1_{J\setminus I}\ud w,I)^2\Big)^{1/2}\leq\mathcal{T} w(J)^{1/2}.
\end{equation*}
To bound this constant from above, we consider the stronger estimate with $Q^u$ replaced by the pointwise bigger $Q$. We also recall the definition of $\mu_I$. Thus we are led to estimate
\begin{equation*}
\begin{split}
  \Big(\sum_{I:I\subsetneq J} &\mu_I\cdot Q(1_{J\setminus I}\ud w,I)^2\Big)^{1/2} \\
  &=\Big(\sum_{S\in\mathscr{S}}\sum_{\substack{K\in\mathscr{K}_S\\ K\subseteq I^u(K)=:I\subsetneq J}}
      \Norm{\tilde P_{S,K}^{\sigma}\operatorname{id}}{L^2(\sigma)}^2\cdot Q(1_{J\setminus I}\ud w,I)^2\Big)^{1/2}.
\end{split}
\end{equation*}
By duality and the pairwise orthogonality of the projections $\tilde P_{S,K}^{\sigma}$, this expression is equal to the supremum over $\Norm{f}{L^2(\sigma)}\leq 1$ of
\begin{equation*}
  \sum_{S\in\mathscr{S}}\sum_{\substack{K\in\mathscr{K}_S\\ K\subseteq I^u(K)=:I\subsetneq J}}
      \pair{\tilde P_{S,K}^{\sigma}f}{\operatorname{id}}_{\sigma}\cdot Q(1_{J\setminus I}\ud w,I),
\end{equation*}
where we may assume without loss of generality that $\pair{\Delta_I^\sigma f}{\id}_{\sigma}\geq 0$ for all $I$.

By Proposition~\ref{prop:monot}, we deduce that (recall that $I\supseteq 3K$)
\begin{equation*}
\begin{split}
  \pair{\tilde P_{S,K}^{\sigma}f}{\operatorname{id}}_{\sigma}\cdot Q(1_{J\setminus I}\ud w,I)
  &\eqsim B(\tilde P_{S,K}^{\sigma}f,1_{J\setminus I})
  \leq B(\tilde P_{S,K}^{\sigma}f,1_{J\setminus K}) \\
  &=B(\tilde P_{S,K}^{\sigma}f,1_J)-B(\tilde P_{S,K}^{\sigma}f,1_K).
\end{split}
\end{equation*}
The sum over the first terms is easy to estimate:
\begin{equation*}
\begin{split}
  \Babs{\sum_{S\in\mathscr{S}}\sum_{\substack{K\in\mathscr{K}_S\\ K\subseteq I^u(K)\subsetneq J}}B(\tilde P_{S,K}^{\sigma}f,1_J)}
  &=\Babs{B\Big(\sum_{S\in\mathscr{S}}\sum_{\substack{K\in\mathscr{K}_S\\ K\subseteq I^u(K)\subsetneq J}}\tilde P_{S,K}^{\sigma}f,1_J\Big)} \\
  &\leq\mathcal{H}^*\BNorm{\sum_{S\in\mathscr{S}}\sum_{\substack{K\in\mathscr{K}_S\\ K\subseteq I^u(K)\subsetneq J}}\tilde P_{S,K}^{\sigma}f}{L^2(\sigma)}w(J)^{1/2} \\
  &\leq\mathcal{H}^*\Norm{f}{L^2(\sigma)}w(J)^{1/2}.
\end{split}
\end{equation*}

We turn our attention to the sum
\begin{equation}\label{eq:testTtoFinish}
  \sum_{S\in\mathscr{S}}\sum_{\substack{K\in\mathscr{K}_S\\ K\subseteq I^u(K)\subsetneq J}}B(\tilde P_{S,K}^{\sigma}f,1_K)
  =\sum_{\substack{S\in\mathscr{S}\\ S\subseteq J}}+\sum_{\substack{S\in\mathscr{S}\\ S\not\subseteq J}}=:I+II.
\end{equation}
For the first part, we have
\begin{equation*}
\begin{split}
  I  &\leq\sum_{\substack{S\in\mathscr{S} \\ S\subseteq J }}\sum_{K\in\mathscr{K}_S}\mathcal{H}^*
      \Norm{\tilde P_{S,K}^{\sigma}f}{L^2(\sigma)}w(K)^{1/2} \\
   &\leq\mathcal{H}^*\Norm{f}{L^2(\sigma)}\Big(\sum_{\substack{S\in\mathscr{S} \\ S\subseteq J }}\sum_{K\in\mathscr{K}_S}
         w(K)\Big)^{1/2}.
\end{split}
\end{equation*}
Using the disjointness of the intervals $K\in\mathscr{K}_S$ and the Carleson property \eqref{eq:Carleson} of the stopping intervals, we find that
\begin{equation*}
  \sum_{\substack{S\in\mathscr{S} \\ S\subseteq J }}\sum_{K\in\mathscr{K}_S}w(K)
  \leq\sum_{\substack{S\in\mathscr{S} \\ S\subseteq J }} w(S)\lesssim w(J).
\end{equation*}

It only remains to consider part $II$ from \eqref{eq:testTtoFinish}. 
Let
\begin{equation*}
  \mathscr{K}^J:=\{K\in\mathscr{D}:K\subseteq I^u(K)\subsetneq J,\ \exists S\in\mathscr{S}:S\not\subseteq J, K\in\mathscr{K}_S\}
\end{equation*}
be the collection of all intervals $K$ that appear in this sum.

\begin{lemma}\label{lem:KJbdOverlap}
The intervals in $\mathscr{K}^J$ have bounded overlap.
\end{lemma}

\begin{proof}
For the containment properties of the relevant intervals, let us note that all other intervals in the argument, except possibly $J$, belong to the dyadic system $\mathscr{D}$.

Let $K\in\mathscr{K}^J$; we show that there are at most boundedly many other $K'\in\mathscr{K}^J$ with $K'\supsetneq K$. By definition, there exist some $S,S'\in\mathscr{S}$ such that $S,S'\not\subseteq J$ and $K\in\mathscr{K}_S$, $K'\in\mathscr{K}_{S'}$. Since $S,S'$ are the minimal intervals in $\mathscr{S}$ that contain
$K^{(r)},(K')^{(r)}$ (respectively),
and $K\subsetneq K'$, it must be that $S\subseteq S'$. In fact $S\subsetneq S'$, since the intervals in $\mathscr{K}_S$ are pairwise disjoint.

Now both $K'$ and $S$ contain $K$, so they intersect, and thus either $K'\subsetneq S$ or $S\subseteq K'$. The latter cannot be, since it would imply that $S\subseteq K'\subseteq J$, contradicting $S\not\subseteq J$. So we know that $K'\subsetneq S\subsetneq S'$.

Thus, we have $K'\subsetneq S$, but $\pi((K')^{(r)})=S'$, so that $(K')^{(r)}\not\subseteq S$. Thus we have that $(K')^{(j)}=S$ for some $j\in\{1,\ldots,r-1\}$. Adding the condition that $K'\supsetneq K$, we see that there are at most $r-1$ different $K'$, and thus the intervals of $\mathscr{K}^J$ have a bounded overlap of at most $r$.
\end{proof}

The estimate for part $II$ of \eqref{eq:testTtoFinish} is now completed by
\begin{equation*}
\begin{split}
  \abs{II} &\leq \sum_{K\in\mathscr{K}^J} \Babs{B\Big(\sum_{\substack{S\in\mathscr{S} \\ S\not\subseteq J \\ \mathscr{K}_S \owns K}}\tilde P_{S,K}^{\sigma}f,1_K\Big)}
  \leq \sum_{K\in\mathscr{K}^J}\mathcal{H}^*\BNorm{\sum_{\substack{S\in\mathscr{S} \\ S\not\subseteq J \\ \mathscr{K}_S\owns K}}\tilde P_{S,K}^{\sigma}f}{L^2(\sigma)}
      w(K)^{1/2} \\
   &\leq\mathcal{H}^*\Norm{f}{L^2(\sigma)}\Big(\sum_{K\in\mathscr{K}^J}w(K)\Big)^{1/2}
   \lesssim\mathcal{H}^*\Norm{f}{L^2(\sigma)}w(J)^{1/2},
\end{split}
\end{equation*}
where the last step used Lemma~\ref{lem:KJbdOverlap}.

Summarizing this subsection, we have shown that
\begin{equation*}
  \mathcal{T}\lesssim\mathcal{H}^*.
\end{equation*}
Feeding this, and the estimates for the other testing constant $\mathcal{T}^*$ and $\mathcal{U}$, back to Corollary~\ref{cor:tailToComplete}, we have established:

\begin{proposition}
\begin{equation*}
  \abs{B_{\operatorname{tail}}(f,g)}
  \lesssim([w,\sigma]_{A_2}^*+\mathcal{H}^*)\Norm{f}{L^2(\sigma)}\Norm{g}{L^2(w)}.
\end{equation*}
\end{proposition}

\section{The local form}

We are left to deal with the local form
\begin{equation*}
  B_{\operatorname{local}}(f,g)
  =\sum_{S\in\mathscr{S}}B_{\operatorname{below}}(\tilde P_S^\sigma f,P_S^w g)
\end{equation*}
from \eqref{eq:Blocal}. It is only here that a particular choice of the stopping intervals $\mathscr{S}$ becomes relevant, and we begin by specifying their construction.

Let $S_0:=I_0$ be the initial interval that contains the supports of $f$ and $g$, and $\mathscr{S}_0:=\{S_0\}$. In the inductive step, given one of the minimal stopping cubes $S$ already chosen, we let $\operatorname{ch}_{\mathscr{S}}(S)$ consist of all maximal dyadic $S'\subsetneq S$ which satisfy at least one of the following:
\begin{equation}\label{eq:massStop}
  \frac{1}{w(S')}\int_{S'}\abs{g}\ud w>4\frac{1}{w(S)}\int_S\abs{g}\ud w
\end{equation}
or
\begin{equation}\label{eq:energyStop}
  \frac{1}{w(S')}\int_{S'}\abs{H(1_S\ud w)}^2\ud \sigma>4\frac{1}{w(S)}\int_S\abs{H(1_S\ud w)}^2\ud \sigma.
\end{equation}
It is immediate that
\begin{equation*}
  \sum_{S'\in\operatorname{ch}_{\mathscr{S}}(S)}w(S')\leq\frac12 w(S),
\end{equation*}
which in turn implies the Carleson condition \eqref{eq:Carleson}, or more precisely,
\begin{equation*}
  \sum_{\substack{S\in\mathscr{S}\\ S\subseteq Q}}w(S)\leq 2w(Q).
\end{equation*}

\begin{remark}
For reasons of comparison, we point out that \eqref{eq:energyStop} is the analogue of the ``energy'' stopping condition of Lacey et al., formulated (after permuting the roles of $\sigma$ and $w$ to be in line with the present treatment) as
\begin{equation*}
  \frac{1}{w(S')}Q(1_{S}\ud w,S')^2\Norm{P^\sigma_{S'}\id}{L^2(\sigma)}^2>C(\mathcal{H}^*)^2
\end{equation*}
A version better suited for our purposes would have $1_{S\setminus S'}$ in place of $1_S$. But then
\begin{equation*}
\begin{split}
  Q(1_{S\setminus S'}\ud w,S')\Norm{P_{S'}^\sigma\id}{L^2(\sigma)}
  &\lesssim\Norm{P_{S'}^\sigma H(1_{S\setminus S'}\ud w)}{L^2(\sigma)} \\
  &\leq\Norm{1_{S'} H(1_{S}\ud w)}{L^2(\sigma)}+\Norm{1_{S'}H(1_{S'}\ud w)}{L^2(\sigma)},
\end{split}  
\end{equation*}
where the first bound follows from \eqref{eq:monot2} with $J=S'$ and $h=1_{S\setminus S'}$ by dualizing with $f\in L^2(\sigma)$, using that $P_{S'}^\sigma$ is self-adjoint and  $f\mapsto\sum_I\epsilon_I\Delta_I^\sigma f$, $\epsilon_I=\pm 1$, is isometric on $L^2(\sigma)$. Noting that
\begin{equation*}
  \Norm{1_{S}H(1_{S}\ud w)}{L^2(\sigma)}\leq\mathcal{H}^* w(S)^{1/2},
\end{equation*}
also with $S'$ in place of $S$, the relatedness of the two stopping conditions is apparent. While they are not strictly comparable, we find that \eqref{eq:energyStop} is both fully operational for our purposes, and perhaps more transparent as a stopping condition familiar from well-known local $Tb$ arguments.
\end{remark}

In view of the mentioned properties of the stopping cubes, the analysis of the form $B_{\operatorname{local}}(f,g)$ splits into the parts $B_{\operatorname{below}}(\tilde P_S^\sigma f,P_S^w g)$, and it suffices to prove the estimate
\begin{equation}\label{eq:localToProve}
 \abs{B_{\operatorname{below}}(\tilde P_S^\sigma f,P_S^w g)}
 \lesssim \mathcal{L}\Norm{\tilde P_S^\sigma f}{L^2(\sigma)}\big(\Norm{P_S^w g}{L^2(w)}+w(S)^{1/2}\ave{\abs{g}}_S^w\big),
\end{equation}
since this is summable over $S\in\mathscr{S}$ to the required bound $\mathcal{L}\Norm{f}{L^2(\sigma)}\Norm{g}{L^2(w)}$. The analysis of this part does not present any essentially new ideas, but we repeat a part of the argument of Lacey \cite{Lacey:2wHilbert} for completeness. (It is interesting, however, that the present streamlined form of the argument became available only after the first publicly distributed version of the present paper.)

Now
\begin{equation*}
  B_S(f,g):=B_{\operatorname{below}}(\tilde P_S^\sigma f,P_S^w g)
  =\sum_{\substack{I,J\subsetneq S; I\Subset J\\ \pi (I^{(r)})=\pi (J^{(1)})=S}}B(\Delta_I^\sigma f,1_J)\ave{\Delta_{J^{(1)}}^w g}_{J}.
\end{equation*}
Actually, the condition that $\pi(J^{(1)})=S$ may be dropped, as this follows from $\pi (I^{(r)})=S$ and $I^{(r)}\subseteq J\subsetneq J^{(1)}\subseteq S$ anyway.
Writing $1_J=1_S-1_{S\setminus J}$, we split $B_S(f,g)$ into two parts, say $B_S^0(f,g)$ and $B_S^1(f,g)$.

\begin{lemma}
For $B_S^0$ defined like $B_S$, but with $1_S$ in place of $1_J$, we have
\begin{equation*}
  \abs{B_S^0(f,g)}\leq 5\mathcal{H}^*\cdot\Norm{\tilde P_S^\sigma f}{L^2(\sigma)}\cdot \ave{\abs{g}}_S^w w(S)^{1/2}.
\end{equation*}
\end{lemma}

\begin{proof}
We have
\begin{equation*}
  B_S^0(f,g)
  =\sum_{I: \pi (I^{(r)})=S}B(\Delta_I^\sigma f,1_S)\sum_{J:I\Subset J\subsetneq S}\ave{\Delta_{J^{(1)}}^w g}_{J}.
\end{equation*}
If non-empty, the inner sum collapses to $\ave{g}_{J(I)}^w-\ave{g}_S^w$ for some $J(I)\subsetneq S$ with $\pi J(I)=S$. By the stopping condition \eqref{eq:massStop}, it follows that
\begin{equation*}
  \ave{g}_{J(I)}^w-\ave{g}_S^w = 5\ave{\abs{g}}_S^w \epsilon_I,\qquad\abs{\epsilon_I}\leq 1,
\end{equation*}
so that
\begin{equation*}
\begin{split}
  \abs{B_S^0(f,g)} &=5\ave{\abs{g}}_S^w \Babs{ B\Big(\sum_{I:\pi(I^{(r)})=S}\epsilon_I\Delta_I^\sigma f,1_S\Big)} \\
  &\leq 5\ave{\abs{g}}_S^w \mathcal{H}^*\BNorm{\sum_{I:\pi(I^{(r)})=S}\epsilon_I\Delta_I^\sigma f}{L^2(\sigma)}w(S)^{1/2} \\
  &\leq 5\mathcal{H}^*\cdot\Norm{\tilde P_S^\sigma f}{L^2(\sigma)}\cdot \ave{\abs{g}}_S^w w(S)^{1/2}.\qedhere
\end{split}
\end{equation*}
\end{proof}

The above bound for $B^{0}_S(f,g)$ is of the correct form required for \eqref{eq:localToProve}, so we are left with estimating the form
\begin{equation}\label{eq:onlyLeftWith}
\begin{split}
  B_S^1(f,g)
  &:=\sum_{\substack{I,J\subsetneq S; I\Subset J\\ \pi (I^{(r)})=S}}B(\Delta_I^\sigma f,1_{S\setminus J})\ave{\Delta_{J^{(1)}}^w g}_{J}\\
  &\phantom{:}=\sum_{\substack{I,J\subsetneq S; I\Subset J\\ \pi I\subsetneq\pi(I^{(r)})=S}}+\sum_{\substack{I,J\subsetneq S; I\Subset J\\ \pi I=S}}
  =:B_S^{10}(f,g)+B_S^{11}(f,g).
\end{split}
\end{equation}

\begin{lemma}
\begin{equation*}
   \abs{B^{10}_S(f,g)}\leq 10\mathcal{H}^*\Norm{\tilde P_S^\sigma f}{L^2(\sigma)}w(S)^{1/2}\ave{\abs{g}}_S^w.
\end{equation*}
\end{lemma}

\begin{proof}
The condition that $\pi I\subsetneq \pi(I^{(r)})=S$ implies that $I\subseteq S'$ for 
  some stopping child $S'\in\operatorname{ch}_{\mathscr{S}}(S)$. Thus
\begin{equation*}
  B_S^{10}(f,g)
  =   \sum_{S'\in\operatorname{ch}_{\mathscr{S}}(S)}
     \sum_{\substack{I:I\subseteq S',\\ \pi(I^{(r)})=S }} B\Big(\Delta_I^\sigma f,\sum_{J:I\Subset J\subsetneq S} 1_{S\setminus J}\ave{\Delta_{J^{(1)}}^w g}_J\Big),
\end{equation*}
where, writing $1_{S\setminus J}=\sum_{L:J\subseteq L\subsetneq S}1_{L^{(1)}\setminus L}$, reorganizing and telescoping,
\begin{equation*}
   \sum_{J:I\Subset J\subsetneq S} 1_{S\setminus J}\ave{\Delta_{J^{(1)}}^w g}_J
   =\sum_{L:I\Subset L\subsetneq S} 1_{L^{(1)}\setminus L}(\ave{g}_L^w-\ave{g}_S^w),
\end{equation*}
is bounded by $1_{S\setminus I^{(r)}}\cdot 5\ave{\abs{g}}_S^w\leq 1_{S\setminus S'}\cdot 5\ave{\abs{g}}_S^w$ by the stopping condition, since $\pi L=S$ for all relevant $L$ here, and $I\subseteq S'$ and $\pi(I^{(r)})=S\supsetneq S'$ imply that $S'\subsetneq I^{(r)}$. From the monotonicity property, Proposition~\ref{prop:monot}, it follows that
\begin{equation*}
  \Babs{ B\Big(\Delta_I^\sigma f,\sum_{J:I\Subset J\subsetneq S} 1_{S\setminus J}\ave{\Delta_{J^{(1)}}^w g}_J\Big)}
  \leq B(\epsilon_I\Delta_I^\sigma f,1_{S\setminus S'})\cdot 5\ave{\abs{g}}_S^w.
\end{equation*}
Abbreviating (note that the collection of intervals on which we project below is a subset of those involved in the projection $\tilde P_S^\sigma$)
\begin{equation*}
  \hat P_{S'}^\sigma f:=\sum_{\substack{I:I\subseteq S', \\ \pi(I^{(r)})=S}}\Delta_I^\sigma f,
\end{equation*}
and recalling the notation $\tilde{f}:=\sum_{I\in\mathscr{D}}\epsilon_I\Delta_I^\sigma f$,
we can thus estimate
\begin{equation*}
\begin{split}
  \abs{B_S^{10}(f,g)} &\leq\sum_{S'\in\operatorname{ch}_{\mathscr{S}}(S)} B(\hat P_{S'}^\sigma \tilde{f}, 1_{S\setminus S'})\cdot 5\ave{\abs{g}}_{S}^w \\
   & = \Big[B\Big( \sum_{S'\in\operatorname{ch}_{\mathscr{S}}(S)}  \hat P_{S'}^\sigma \tilde{f} , 1_S\Big)
           - \sum_{S'\in\operatorname{ch}_{\mathscr{S}}(S)} B( \hat P_{S'}^\sigma \tilde{f} , 1_{S'})\Big]\cdot 5\ave{\abs{g}}_{S}^w \\
    &\leq\mathcal{H}^*\Big[\BNorm{\sum_{S'\in\operatorname{ch}_{\mathscr{S}}(S)}  \hat P_{S'}^\sigma \tilde{f}}{L^2(\sigma)}w(S)^{1/2} \\
     &\qquad\qquad   +\sum_{S'\in\operatorname{ch}_{\mathscr{S}}(S)} \Norm{ \hat P_{S'}^\sigma \tilde{f}}{L^2(\sigma)}w(S')^{1/2}\Big] \cdot 5\ave{\abs{g}}_{S}^w \\
    &\leq\mathcal{H}^*\Big[\Norm{\tilde P_S^\sigma f}{L^2(\sigma)}w(S)^{1/2} \\
    &\qquad\qquad+\Big(\sum_{S'\in\operatorname{ch}_{\mathscr{S}}(S)}\Norm{\hat P_{S'}^\sigma \tilde f}{L^2(\sigma)}^2\Big)^{1/2}
       \Big(\sum_{S'\in\operatorname{ch}_{\mathscr{S}}(S)}w(S')\Big)^{1/2}\Big]\cdot 5\ave{\abs{g}}_S^w \\
     &\leq 10\mathcal{H}^*\Norm{\tilde P_S^\sigma f}{L^2(\sigma)}w(S)^{1/2}\ave{\abs{g}}_S^w.\qedhere
\end{split}
\end{equation*}
\end{proof}

The remaining core estimate is then concerned with bounding
\begin{equation*}
  B^{11}_S(f,g)
  =\sum_{\substack{I,J\subsetneq S; I\Subset J\\ \pi I=S}}B(\Delta_I^\sigma f,1_{S\setminus J})\ave{\Delta_{J^{(1)}}^w g}_{J}.
\end{equation*}
This is a special case of a more general form
\begin{equation*}
  B_{\mathscr{Q}}(f,g)
  :=\sum_{(I,J)\in\mathscr{Q}}B(\Delta_I^\sigma f,1_{S\setminus J})\ave{\Delta_{J^{(1)}}^w g}_{J}
\end{equation*}
related to an \emph{admissible} collection $\mathscr{Q}$ of pairs of intervals, defined as follows:

\begin{definition}[Admissibility in the sense of Lacey \cite{Lacey:2wHilbert}]
A collection $\mathscr{Q}$ of pairs $(I,J)$ of intervals is said to be \emph{admissible} if $I$ and $J^{(1)}$ are good, and moreover
\begin{enumerate}
  \item $I\Subset J\subsetneq S$ (a fixed super-interval) for all $(I,J)\in\mathscr{Q}$, and
  \item if $(I,J_i)\in\mathscr{Q}$ for $i=1,2$, and $J$ is another interval with $J^{(1)}$ good and $J_1\subsetneq J\subsetneq J_2$, then also $(I,J)\in\mathscr{Q}$.
\end{enumerate}
\end{definition}

Indeed, it is immediate to check that
\begin{equation}\label{eq:particularQ}
  \mathscr{Q}=\{(I,J):I,J\subsetneq S; I,J^{(1)}\text{ good}; I\Subset J; \pi I=\pi (J^{(1)})=S\}
\end{equation}
is admissible, and we may freely insert the requirement of goodness into the summation conditions in \eqref{eq:onlyLeftWith}, recalling that the functions $f$ and $g$ are good anyway.

\begin{remark}
The definition given by Lacey \cite[Definition 3.3]{Lacey:2wHilbert} involves a third property, but this has only a single isolated application, so we prefer not to include it in the above definition. Besides, this third property is connected with Lacey's version of the stopping condition \eqref{eq:energyStop}, which differs from the present one.
\end{remark}

We borrow the following key result from Lacey \cite{Lacey:2wHilbert}, the main estimate of that paper. As promised in the Introduction, we can simply apply this bound in the present setting as a black box; it is (almost) unnecessary to visit the details of the proof. We only need to remark that this result, by inspection of Lacey's argument, remains valid with the definition of admissibility as formulated above; the third property imposed by Lacey \cite{Lacey:2wHilbert} is only used when estimating the quantity $\operatorname{size}(\mathscr{Q})$.

\begin{proposition}[(Lacey \cite{Lacey:2wHilbert})]\label{prop:LaceyMain}
For any admissible collection $\mathscr{Q}$, we have
\begin{equation*}
\begin{split}
  \abs{B_{\mathscr{Q}}(f,g)}
  &\lesssim\operatorname{size}(\mathscr{Q})\Norm{f}{L^2(\sigma)}\Norm{g}{L^2(w)},\\
  \operatorname{size}(\mathscr{Q})^2 &:=
   \sup_{K\in\mathscr{I}\cup\mathscr{J}}\frac{1}{w(K)}Q(1_{S\setminus K}\ud w,K)^2
     \sum_{\substack{I\subseteq K \\ I\in\mathscr{I}}}\pair{\id}{h^\sigma_I}_{\sigma}^2
\end{split}
\end{equation*}
where $\mathscr{I}:=\{I:(I,J)\in\mathscr{Q}\text{ for some }J\}$ and $\mathscr{J}:=\{J:(I,J)\in\mathscr{Q}\text{ for some }I\}$.
\end{proposition}

This is not explicitly formulated in \cite{Lacey:2wHilbert}, but its essence is contained in the iterative estimate of \cite[Lemma 3.6]{Lacey:2wHilbert}, and the discussion around that lemma.

To complete the estimation of the local part, and thereby of the Hilbert transform altogether, we only need to add the observation that
\begin{equation*}
  B_S^{11}(f,g)=B_S^{11}(\tilde P_S^\sigma f, P_S^w g),
\end{equation*}
and the following upper bound for $\operatorname{size}(\mathscr{Q})$. It is here that Lacey uses his third condition from the definition of admissibility, which is here replaced by an application of our stopping condition \eqref{eq:energyStop}.

\begin{lemma}\label{lem:estSize}
For the particular collection $\mathscr{Q}$ in \eqref{eq:particularQ}, we have
\begin{equation*}
  \operatorname{size}(\mathscr{Q})
  \leq 3\mathcal{H}^*.
\end{equation*}
\end{lemma}

\begin{proof}
Note that
\begin{equation*}
\begin{split}
  Q(1_{S\setminus K} &\ud w,K)
    \Big( \sum_{\substack{I\in\mathscr{I} \\ I\subseteq K }}\pair{\id}{h^\sigma_I}_{\sigma}^2\Big)^{1/2} \\
 &= Q(1_{S\setminus K}\ud w,K) \sup_{\Norm{f}{L^2(\sigma)}\leq 1}
     \Bpair{\sum_{\substack{I\in\mathscr{I} \\ I\subseteq K}}\Delta_I^\sigma f}{\id} \\
  &\lesssim \sup_{\Norm{f}{L^2(\sigma)}\leq 1}B\Big(\sum_{\substack{I\in\mathscr{I} \\ I\subseteq K }}\Delta_I^\sigma f,1_{S\setminus K}\Big) \\
  &\leq\Norm{1_K H(1_{S\setminus K}\ud w)}{L^2(\sigma)}  \\
  &\leq\Norm{1_K H(1_{S}\ud w)}{L^2(\sigma)}+\Norm{1_K H(1_{K}\ud w)}{L^2(\sigma)}.
\end{split}
\end{equation*}
The second term is bounded by $\mathcal{H}^* w(K)^{1/2}$, as a direct application of interval testing. For the first term, we have
\begin{equation*}
\begin{split}
  \Norm{1_K H(1_{S}\ud w)}{L^2(\sigma)}
  &=w(K)^{1/2}\Big(\frac{1}{w(K)}\int_K\abs{H(1_S\ud w)}^2\ud\sigma\Big)^{1/2} \\
  &\leq w(K)^{1/2}\Big(\frac{4}{w(S)}\int_S\abs{H(1_S\ud w)}^2\ud\sigma\Big)^{1/2}
  \leq 2\mathcal{H}^*w(K)^{1/2}
\end{split}
\end{equation*}
by the stopping condition \eqref{eq:energyStop} (using $\pi K=S$) and interval testing.
Thus
\begin{equation*}
  Q(1_{S\setminus K}\ud w,K)^2
    \sum_{\substack{I\in\mathscr{I} \\ I\subseteq K}}\pair{\id}{h^\sigma_I}_{\sigma}^2
  \leq(\mathcal{H}^*w(K)^{1/2}+2\mathcal{H}^*w(K)^{1/2})^2= (3\mathcal{H}^*)^2 w(K),
\end{equation*}
and the claim follows.
\end{proof}

Proposition~\ref{prop:LaceyMain}, Lemma~\ref{lem:estSize} and the preceding discussion prove that
\begin{equation*}
  \abs{B_{\operatorname{local}}(f,g)}\lesssim\mathcal{H}^*\Norm{f}{L^2(\sigma)}\Norm{g}{L^2(w)},
\end{equation*}
and this completes the estimation of the bilinear form of the Hilbert transform altogether, with the desired bound of the form $\mathcal{H}+\mathcal{H}^*+[\sigma,w]_{A_2}^*+[w,\sigma]_{A_2}^*$. The proof of Theorem~\ref{thm:main} is complete.


\end{document}